\documentclass[a4paper, 11pt, dvipdfmx]{article}
\usepackage[top=30truemm, bottom=30truemm, left=20truemm, right=20truemm]{geometry}
\usepackage{amsthm}
\usepackage{amsmath}
\usepackage{amssymb}
\usepackage{amsfonts}
\usepackage{mathtools}
\usepackage{bm}
\usepackage{color}
\usepackage[dvipdfm,
	draft=false,
	bookmarks=true,
	bookmarksnumbered=true,
	bookmarksopen=false,
	bookmarksopenlevel=2,
	colorlinks=false,
pdfborder={000}]{hyperref}
\newtheoremstyle{mystyle}
	{\baselineskip}
	{\baselineskip}
	{\itshape}
	{}
	{\bfseries}
	{.}
	{1em}
	{}

\theoremstyle{mystyle}
\newtheorem{thm}{Theorem}[section]
\newtheorem{lem}[thm]{Lemma}
\newtheorem{cor}[thm]{Corollary}
\newtheorem{prop}[thm]{Proposition}

\newtheoremstyle{mystyle2}
	{\baselineskip}
	{\baselineskip}
	{\normalfont}
	{}
	{\bfseries}
	{.}
	{1em}
	{}

\theoremstyle{mystyle2}
	\newtheorem{rem}[thm]{Remark}

\numberwithin{equation}{section}
\numberwithin{table}{section}
\numberwithin{figure}{section}
\allowdisplaybreaks[4]
\begin{document}
\title{
	Weighted estimates and large time behavior of small amplitude solutions to the semilinear heat equation
}
\author{
	Ryunosuke Kusaba \footnote{e-mail : ryu2411501@akane.waseda.jp} \bigskip \\
	Department of Pure and Applied Physics, \\
	Graduate School of Advanced Science and Engineering, \\
	Waseda University, 3-4-1 Okubo, Shinjuku, Tokyo 169-8555, JAPAN \bigskip \\
	Tohru Ozawa \bigskip \\
	Department of Applied Physics, \\
	Waseda University, 3-4-1 Okubo, Shinjuku, Tokyo 169-8555, JAPAN
}
\date{}
\maketitle
\begin{abstract}
	We present a new method to obtain weighted $L^{1}$-estimates of global solutions to the Cauchy problem for the semilinear heat equation with a simple power of super-critical Fujita exponent.
	Our approach is based on direct and explicit computations of commutation relations between the heat semigroup and monomial weights in $\mathbb{R}^{n}$, while it is independent of the standard parabolic arguments which rely on the comparison principle or some compactness arguments.
	We also give explicit asymptotic profiles with parabolic self-similarity of the global solutions.
\end{abstract}



\newpage
\section{Introduction}
We study the large time behavior of global solutions to the Cauchy problem for the semilinear heat equations of the form
\begin{align}
	\tag{P} \label{P}
	\begin{cases}
		\partial_{t} u- \Delta u=f \left( u \right), &\quad \left( t, x \right) \in \left( 0, + \infty \right) \times \mathbb{R}^{n}, \\
		u \left( 0 \right) = \varphi, &\quad x \in \mathbb{R}^{n},
	\end{cases}
\end{align}
where $u \colon \left[ 0, + \infty \right) \times \mathbb{R}^{n} \to \mathbb{R}$ is an unknown function, $\Delta$ is the Laplacian in $\mathbb{R}^{n}$, $\varphi \colon \mathbb{R}^{n} \to \mathbb{R}$ is a given data at $t=0$, and $f \colon \mathbb{R} \to \mathbb{R}$ is a nonlinear term such that there exist constants $K>0$ and $p \in \left( 1, + \infty \right)$ that ensure the estimate
\begin{align}
	\label{eq:esti_f}
	\left\lvert f \left( \xi \right) -f \left( \eta \right) \right\rvert \leq K \left( \left\lvert \xi \right\rvert^{p-1} + \left\lvert \eta \right\rvert^{p-1} \right) \left\lvert \xi - \eta \right\rvert
\end{align}
for all $\xi, \eta \in \mathbb{R}$ and that $f \left( 0 \right) =0$.
Typical examples are given by homogeneous functions of order $p$ such as
\begin{align*}
	f \left( \xi \right) = \pm \xi^{p}, {\,} \pm \left\lvert \xi \right\rvert^{p}, {\,} \pm \left\lvert \xi \right\rvert^{p-1} \xi.
\end{align*}

The behavior of solutions to \eqref{P} has been studied by many mathematicians since the pioneering work by Fujita \cite{Fujita} and it is revealed that the behavior changes depending on the exponent $p$ in the nonlinear term, the size of the initial data $\varphi$, and so on.
In particular, the exponent $p_{\mathrm{F}} \left( n \right) \coloneqq 1+2/n$, called the Fujita exponent, gives a threshold that characterizes the large time behavior of solutions to \eqref{P}.
In the case of super-critical Fujita exponent, namely, $p>p_{\mathrm{F}} \left( n \right)$, it is known that solutions to \eqref{P} with small initial data behave like the solution to the linear heat equation: \eqref{P} with $f \equiv 0$ (cf. \cite{Gmira-Veron, Kawanago1996}).
In this paper, we consider the large time behavior of global solutions to \eqref{P} for small initial data in $L^{1} \left( \mathbb{R}^{n} \right)$ with algebraic weights in the case where $p>p_{\mathrm{F}} \left( n \right)$.

To state our main results precisely, we introduce some notation.
Let $\left( e^{t \Delta}; t \geq 0 \right)$ be the heat semigroup given by
\begin{align*}
	e^{t \Delta} \varphi = \begin{cases}
		G_{t} \ast \varphi, &\qquad t>0, \\
		\varphi, &\qquad t=0
	\end{cases}
\end{align*}
for $\varphi \in L^{q} \left( \mathbb{R}^{n} \right)$ with $q \in \left[ 1, + \infty \right]$, where $G_{t} \colon \mathbb{R}^{n} \to \mathbb{R}$ is the Gauss kernel given by
\begin{align*}
	G_{t} \left( x \right) = \left( 4 \pi t \right)^{- \frac{n}{2}} \exp \left( - \frac{\left\lvert x \right\rvert^{2}}{4t} \right), \qquad x \in \mathbb{R}^{n},
\end{align*}
$\ast$ is the convolution in $\mathbb{R}^{n}$, and $L^{q} \left( \mathbb{R}^{n} \right)$ is the standard Lebesgue space with the norm denoted by $\left\lVert {\,\cdot\,} \right\rVert_{q}$.
We also need the weighted $L^{1}$-space defined by
\begin{align*}
	L_{m}^{1} \left( \mathbb{R}^{n} \right) = \left\{ \varphi \in L^{1} \left( \mathbb{R}^{n} \right); {\,} \text{$x^{\alpha} \varphi \in L^{1} \left( \mathbb{R}^{n} \right)$ for all $\alpha \in \mathbb{Z}_{\geq 0}^{n}$ with $\left\lvert \alpha \right\rvert \leq m$} \right\},
\end{align*}
where $x^{\alpha} \varphi$ means the function $\mathbb{R}^{n} \ni x \mapsto x^{\alpha} \varphi \left( x \right) \in \mathbb{R}$.

We start with the most basic result on the existence and uniqueness of global solutions to \eqref{P} in the function space $X$ defined by
\begin{align*}
	X= \left( C \cap L^{\infty} \right) \left( \left[ 0, + \infty \right); L^{1} \left( \mathbb{R}^{n} \right) \right) \cap \left( C \cap L^{\infty} \right) \left( \left( 0, + \infty \right); L^{\infty} \left( \mathbb{R}^{n} \right) \right).
\end{align*}

\begin{prop} \label{pro:P_global}
	Let $p>p_{\mathrm{F}} \left( n \right)$.
	Then, there exists $\varepsilon_{0} >0$ such that for any $\varphi \in \left( L^{1} \cap L^{\infty} \right) \left( \mathbb{R}^{n} \right)$ with $\left\lVert \varphi \right\rVert_{1} + \left\lVert \varphi \right\rVert_{\infty} \leq \varepsilon_{0}$, \eqref{P} has a unique global solution $u \in X$, which satisfies
	\begin{align}
		\label{eq:P_decay}
		\sup_{q \in \left[ 1, + \infty \right]} \sup_{t \geq 0} \left( 1+t \right)^{\frac{n}{2} \left( 1- \frac{1}{q} \right)} \left\lVert u \left( t \right) \right\rVert_{q} <+ \infty.
	\end{align}
\end{prop}

Although the above proposition is more or less well-known, we give the proof in Appendix \ref{app:A} to make this paper self-contained (see also \cite[Theorem 1.2]{Ishige-Kawakami-Kobayashi-1} and \cite[Theorem 20.15]{Quittner-Souplet}).
We remark that any additional conditions for the initial data such as nonnegativity or exponential decay at the far field are not supposed.

Based on Proposition \ref{pro:P_global}, we show that $L^{1}$-space with algebraic weights is invariant under the semilinear heat flow associated with \eqref{P}.

\begin{thm} \label{th:P_weight}
	Let $p>p_{\mathrm{F}} \left( n \right)$ and let $m \in \mathbb{Z}_{>0}$.
	Let $\varphi \in \left( L_{m}^{1} \cap L^{\infty} \right) \left( \mathbb{R}^{n} \right)$ satisfy $\left\lVert \varphi \right\rVert_{1} + \left\lVert \varphi \right\rVert_{\infty} \leq \varepsilon_{0}$ and let $u \in X$ be the global solution to \eqref{P} given in Proposition \ref{pro:P_global}.
	Then, $u \in C \left( \left[ 0, + \infty \right); L_{m}^{1} \left( \mathbb{R}^{n} \right) \right)$, and moreover there exists $C_{m} >0$ such that the estimate
	\begin{align}
		\sum_{\left\lvert \alpha \right\rvert =m} \left\lVert x^{\alpha} u \left( t \right) \right\rVert_{1} \leq C_{m} \left( 1+t^{\frac{m}{2}} \right)
	\end{align}
	holds for all $t \geq 0$.
\end{thm}

The next theorem describes the large time behavior of the global solution to \eqref{P} given in Proposition \ref{pro:P_global}.

\begin{thm} \label{th:P_appro}
	Let $p>p_{\mathrm{F}} \left( n \right)$.
	Let $\varphi \in \left( L^{1} \cap L^{\infty} \right) \left( \mathbb{R}^{n} \right)$ satisfy $\left\lVert \varphi \right\rVert_{1} + \left\lVert \varphi \right\rVert_{\infty} \leq \varepsilon_{0}$ and let $u \in X$ be the global solution to \eqref{P} given in Proposition \ref{pro:P_global}.
	Then, for any $N \in \mathbb{Z}_{>0}$ and $q \in \left[ 1, + \infty \right]$, there exists $C_{N, q} >0$ such that the estimates
	\begin{align}
		t^{\frac{n}{2} \left( 1- \frac{1}{q} \right)} \left\lVert u \left( t \right) -u_{N} \left( t \right) \right\rVert_{q} \leq \begin{dcases}
			C_{N, q} t^{-N \sigma} &\quad \text{if} \quad 0<N \sigma <1, \\
			C_{N, q} t^{-1} \log \left( 1+t \right) &\quad \text{if} \quad N \sigma =1, \\
			C_{N, q} t^{-1} &\quad \text{if} \quad N \sigma >1
		\end{dcases}
	\end{align}
	hold for all $t>2^{N-1}$, where
	\begin{align*}
		\sigma &\coloneqq \frac{n}{2} \left( p-1 \right) -1>0, \\
		u_{N} \left( t \right) &\coloneqq \begin{dcases}
			e^{t \Delta} \varphi_{1} &\quad \text{for} \quad N=1, \\
			e^{t \Delta} \varphi_{N} + \int_{0}^{t} e^{\left( t-s \right) \Delta} f \left( u_{N-1} \left( s \right) \right) ds &\quad \text{for} \quad N \geq 2,
		\end{dcases} \\
		\varphi_{N} &\coloneqq \begin{dcases}
			\varphi + \int_{0}^{+ \infty} f \left( u \left( s \right) \right) ds &\quad \text{for} \quad N=1, \\
			\varphi + \int_{0}^{+ \infty} \left( f \left( u \left( s \right) \right) -f \left( u_{N-1} \left( s \right) \right) \right) ds &\quad \text{for} \quad N \geq 2.
		\end{dcases}
	\end{align*}
\end{thm}

Alternatively, the sequence $\left( u_{N}; N \in \mathbb{Z}_{\geq 0} \right)$ in $X$ is introduced recursively by $u_{0} \equiv 0$ and
\begin{align}
	\label{eq:linear_appro}
	\begin{cases}
		\partial_{t} u_{N} - \Delta u_{N} =f \left( u_{N-1} \right), &\quad \left( t, x \right) \in \left( 0, + \infty \right) \times \mathbb{R}^{n}, \\
		u_{N} \left( 0 \right) = \varphi_{N}, &\quad x \in \mathbb{R}^{n}
	\end{cases}
\end{align}
for $N \geq 1$.
In addition, it follows from Proposition \ref{pro:P_global} that
\begin{align*}
	\sup_{q \in \left[ 1, + \infty \right]} \sup_{t \geq 0} \left( 1+t \right)^{\frac{n}{2} \left( 1- \frac{1}{q} \right)} \left\lVert u_{N} \left( t \right) \right\rVert_{q} <+ \infty
\end{align*}
for any $N \in \mathbb{Z}_{>0}$ (see Lemma \ref{lem:P_appro_decay} in Section \ref{sec:4}).
Hence, Theorem \ref{th:P_appro} means that the large time behavior of the global solution to \eqref{P} is approximated by that of the solution to the Cauchy problem for the linear heat equation \eqref{eq:linear_appro}.
In particular, the larger $N \in \mathbb{Z}_{>0}$ is, the faster the remainder vanishes up to $t^{-1}$.
We also emphasize that we do not need weighted $L^{1}$-spaces in Theorem \ref{th:P_appro}.
The sequences $\left( \varphi_{N}; N \in \mathbb{Z}_{>0} \right)$ and $\left( u_{N}; N \in \mathbb{Z}_{>0} \right)$ are constructed as follows.
For $N=1$, we define $u_{1}$ and $\varphi_{1}$ by inference from the well-studied fact that
\begin{gather*}
	\lim_{t \to + \infty} t^{\frac{n}{2} \left( 1- \frac{1}{q} \right)} \left\lVert u \left( t \right) - \left( \int_{\mathbb{R}^{n}} \varphi_{1} \left( y \right) dy \right) G_{t} \right\rVert_{q} =0, \\
	\lim_{t \to + \infty} t^{\frac{n}{2} \left( 1- \frac{1}{q} \right)} \left\lVert e^{t \Delta} \varphi_{1} - \left( \int_{\mathbb{R}^{n}} \varphi_{1} \left( y \right) dy \right) G_{t} \right\rVert_{q} =0
\end{gather*}
hold for any $q \in \left[ 1, + \infty \right]$ under suitable assumptions (see \cite{Gmira-Veron, Kawanago1996, Kawanago1997, Taskinen, Ishige-Kawakami2009, Quittner-Souplet, Ishige-Ishiwata-Kawakami, Ishige-Kawakami2012} and Proposition \ref{pro:heat_asymptotics_0} in Section \ref{sec:Heat_semigroup}).
For $N \geq 2$, we set $\varphi_{N}$ to make the difference $u \left( t \right) -u_{N-1} \left( t \right)$ in the decomposition of the difference $u \left( t \right) -u_{N} \left( t \right)$, where $u_{N}$ is constructed by the iteration.

The following theorem provides us with explicit self-similar profiles of the global solution to \eqref{P} given in Proposition \ref{pro:P_global} with explicit remainder estimates.

\begin{thm} \label{th:P_aymptotics}
	Let $m \in \left\{ 0, 1 \right\}$ and let $p>1+ \left( 3+m \right) /n$.
	Let $\varphi \in \left( L_{m+1}^{1} \cap L^{\infty} \right) \left( \mathbb{R}^{n} \right)$ satisfy $\left\lVert \varphi \right\rVert_{1} + \left\lVert \varphi \right\rVert_{\infty} \leq \varepsilon_{0}$ and let $u \in X$ be the global solution to \eqref{P} given in Proposition \ref{pro:P_global}.
	Then, for any $q \in \left[ 1, + \infty \right]$, there exists $C_{q} >0$ such that the estimates
	\begin{alignat}{2}
		\label{eq:P_asymptotics_0}
		&t^{\frac{n}{2} \left( 1- \frac{1}{q} \right)} \left\lVert u \left( t \right) -c_{0} \delta_{t} G_{1} \right\rVert_{q} \leq C_{q} t^{- \frac{1}{2}} &&\quad \text{for} \quad m=0, \\
		\label{eq:P_asymptotics_1}
		&t^{\frac{n}{2} \left( 1- \frac{1}{q} \right)} {\,} \biggl\lVert u \left( t \right) -c_{0} \delta_{t} G_{1} - \frac{1}{2} t^{- \frac{1}{2}} \sum_{j=1}^{n} c_{j} \delta_{t} \left( x_{j} G_{1} \right) \biggr\rVert_{q} \leq C_{q} t^{-1} &&\quad \text{for} \quad m=1
	\end{alignat}
	hold for all $t>1$, where
	\begin{gather*}
		c_{0} \coloneqq \int_{\mathbb{R}^{n}} \varphi_{1} \left( y \right) dy, \qquad c_{j} \coloneqq \int_{\mathbb{R}^{n}} y_{j} \varphi_{1} \left( y \right) dy, \\
		\varphi_{1} \coloneqq \varphi + \int_{0}^{+ \infty} f \left( u \left( s \right) \right) ds
	\end{gather*}
	and $\delta_{t}$ is the dilation acting on functions $\psi$ on $\mathbb{R}^{n}$ as
	\begin{align*}
		\left( \delta_{t} \psi \right) \left( x \right) =t^{- \frac{n}{2}} \psi \left( t^{- \frac{1}{2}} x \right), \qquad x \in \mathbb{R}^{n}.
	\end{align*}
\end{thm}

\begin{rem}
	Theorems \ref{th:P_weight}, \ref{th:P_appro} and \ref{th:P_aymptotics} hold for global solutions to \eqref{P} in $X$ satisfying \eqref{eq:P_decay} without the smallness assumption for the initial data.
	However, it is known that solutions to \eqref{P} do not necessarily satisfy \eqref{eq:P_decay}.
	In fact, \cite{Kawanago1996} showed that if $f \left( \xi \right) = \xi^{p}$ with $p>p_{\mathrm{F}} \left( n \right)$ and the initial data is nonnegative and sufficiently large, then the solution to \eqref{P} does not satisfy \eqref{eq:P_decay} and even it can be blow up in finite time.
\end{rem}

There is a large literature on asymptotic expansions of global solutions to semilinear heat equations.
For example, the $0$-th order asymptotic expansion like \eqref{eq:P_asymptotics_0} was obtained by \cite{Gmira-Veron, Kawanago1996, Kawanago1997, Taskinen, Ishige-Kawakami2009, Quittner-Souplet} with various methods, and higher order asymptotic expansions were given by \cite{Ishige-Ishiwata-Kawakami, Ishige-Kawakami2012, Ishige-Kawakami2013, Ishige-Kawakami-Kobayashi-2, Ishige-Kawakami-Michihisa}.
Therefore, the above main results seem to be well-known except for Theorem \ref{th:P_appro}, but the novelties of this paper lie in the method of the proofs.
In \cite{Ishige-Ishiwata-Kawakami}, the authors introduced $L^{1}$-decay estimates of the heat semigroup to derive a classification of global solutions to \eqref{P} in terms of decay rate in $t$ of their $L^{1}$-norm with a regularity assumption for the nonlinear term.
This method was improved by \cite{Ishige-Kawakami2012} to obtain higher order asymptotic expansions in the case where $p>p_{\mathrm{F}} \left( n \right)$ at the expense of the parabolic self-similarity of asymptotic profiles and generalized by \cite{Ishige-Kawakami2013, Ishige-Kawakami-Kobayashi-2, Ishige-Kawakami-Michihisa} to apply other semilinear parabolic equations.
This method is also valid to obtain asymptotic expansions of global solutions to semilinear damped wave equations \cite{Kawakami-Ueda, Kawakami-Takeda}.
Therefore, the method introduced and improved by \cite{Ishige-Ishiwata-Kawakami, Ishige-Kawakami2012} is powerful, while it seems to be the only method to obtain not only higher order asymptotic expansions but also the first order asymptotic expansion such as \eqref{eq:P_asymptotics_1}.

In this paper, we show Theorem \ref{th:P_aymptotics} by using the linear approximation given in Theorem \ref{th:P_appro} and asymptotic expansions of the heat semigroup with Hermite polynomials (see Proposition \ref{pro:heat_asymptotics_higher} in Section \ref{sec:Heat_semigroup}).
We emphasize that the asymptotic profiles in Theorem \ref{th:P_aymptotics} have the parabolic self-similarity in the sense that each term has the form of a constant multiple of dilated functions on $\mathbb{R}^{n}$ (with decay factor in $t$).
Comparing Theorem \ref{th:P_aymptotics} and the asymptotic expansions of the heat semigroup, we see that the global solution has the same asymptotic profiles as for the heat semigroup with the initial data $\varphi_{1}$ if $m \in \left\{ 0, 1 \right\}$.
We also see that the solution asymptotically approaches to the same profiles as for the heat semigroup with the initial data $\varphi_{1}$ even if $m \in \mathbb{Z}_{\geq 0}$ with $m \geq 2$.
However, it is not suitable to consider them as higher order asymptotic expansions of the global solution in terms of the decay rates of the remainders (for details, see Remark \ref{rem:P_asymptotics_higher} in Section \ref{sec:4}).

Furthermore, we need weighted $L^{1}$-estimates of the global solution to \eqref{P} as in Theorem \ref{th:P_weight} to ensure the finiteness of constants $c_{0}, c_{1}, \ldots, c_{n}$.
As far as we know, there are two methods to derive the weighted estimates.
One is to attribute the weighted estimates to those of the heat semigroup by using the comparison principle.
The other is to approximate the global solution by solutions to linear heat equations via an iteration argument with the Ascoli-Arzel\`{a} theorem (see \cite[Lemma 3.1]{Ishige-Ishiwata-Kawakami} and \cite[Theorem 1.2]{Ishige-Kawakami-Kobayashi-1}, respectively).
We note that the comparison principle which we use in the first method is proved by a computation of the energy for a positive or negative part of the difference of solutions to heat equations (see \cite[Proposition 52.10]{Quittner-Souplet}).
Therefore, to obtain the weighted estimates by using the above methods, we have to verify the regularity or uniqueness of global solutions to \eqref{P} at the same time.
By the way, the global solution given in Proposition \ref{pro:P_global} is constructed by a contraction argument for the following integral equation associated with \eqref{P}:
\begin{align*}
	\tag{I} \label{I}
	u \left( t \right) &=e^{t \Delta} \varphi + \int_{0}^{t} e^{\left( t-s \right) \Delta} f \left( u \left( s \right) \right) ds.
\end{align*}
Hence, it is a natural question from the view point of a priori estimates to consider what properties the solution constructed by the contraction argument has in the framework of the integral equation without reconstruction (approximation) of the global solution.
We show Theorem \ref{th:P_weight} by direct and explicit calculations with the aid of commutation relations and their estimates between the heat semigroup and monomial weights in $\mathbb{R}^{n}$ given by Theorems \ref{th:heat_commutator} and \ref{th:heat_commutator_esti} in Section \ref{sec:Heat_semigroup}.
Since we do not use the standard parabolic arguments which rely on the comparison principle or some compactness arguments, our method enables us to discuss independently the well-posedness and a priori estimates for \eqref{P}.
Therefore, our method is available to not only parabolic equations but also some models with dispersion which cannot be applied the comparison principle to; for example, the following complex Ginzburg-Landau type equation:
\begin{align*}
	\partial_{t} u- \left( \lambda +i \alpha \right) \Delta u+ \left( \kappa +i \beta \right) \left\lvert u \right\rvert^{q-1} u- \gamma u=0,
\end{align*}
where $u \colon \left[ 0, + \infty \right) \times \mathbb{R}^{n} \to \mathbb{C}$ is an unknown function and $\alpha, \beta, \gamma, \kappa \in \mathbb{R}$, $\lambda >0$, $q \in \left[ 1, + \infty \right]$ are given parameters (cf. \cite{Shimotsuma-Yokota-Yoshii2014, Shimotsuma-Yokota-Yoshii2016}).
Furthermore, our method using commutation relations helps us to understand nonlinear parabolic equations with fundamental techniques in calculus.
From these observation, our approach would have an advantage over the previous studies.
The results on new higher order asymptotic expansions for the complex Ginzburg-Landau type equation which includes \eqref{P} in a special case will be discussed in the forthcoming paper.

This paper is organized as follows.
In Section \ref{sec:Heat_semigroup}, we introduce basic estimates and asymptotic expansions in the linear case.
In Section \ref{sec:P_weight}, we prove Theorem \ref{th:P_weight}.
In Section \ref{sec:4}, we prove Theorems \ref{th:P_appro} and \ref{th:P_aymptotics}.

\section{Heat semigroup} \label{sec:Heat_semigroup}
In this section, we introduce basic estimates and asymptotic expansions of the heat semigroup.
See \cite{Giga-Giga-Saal, Quittner-Souplet} for the standard notion and notation of the subject.
To begin with, we prepare some notation.
Let $\mathbb{Z}_{>0}$ be the set of positive integers and let $\mathbb{Z}_{\geq 0} \coloneqq \mathbb{Z}_{>0} \cup \left\{ 0 \right\}$.
For $\alpha = \left( \alpha_{1}, \ldots, \alpha_{n} \right) \in \mathbb{Z}_{\geq 0}^{n}$ and $x= \left( x_{1}, \ldots, x_{n} \right) \in \mathbb{R}^{n}$, we define
\begin{align*}
	\left\lvert \alpha \right\rvert \coloneqq \sum_{j=1}^{n} \alpha_{j}, \qquad \alpha ! \coloneqq \prod_{j=1}^{n} \alpha_{j} !, \qquad x^{\alpha} \coloneqq \prod_{j=1}^{n} x_{j}^{\alpha_{j}}, \qquad \partial^{\alpha} = \partial_{x}^{\alpha} \coloneqq \prod_{j=1}^{n} \partial_{j}^{\alpha_{j}}, \qquad \partial_{j} \coloneqq \frac{\partial}{\partial x_{j}}.
\end{align*}
For $\alpha = \left( \alpha_{1}, \ldots, \alpha_{n} \right)$, $\beta = \left( \beta_{1}, \ldots, \beta_{n} \right) \in \mathbb{Z}_{\geq 0}^{n}$, $\alpha \leq \beta$ means that $\alpha_{j} \leq \beta_{j}$ holds for any $j \in \left\{ 1, \ldots, n \right\}$.
Furthermore, we write
\begin{align*}
	\binom{\alpha}{\beta} \coloneqq \begin{dcases}
		\frac{\alpha !}{\beta ! \left( \alpha - \beta \right) !} = \prod_{j=1}^{n} \frac{\alpha_{j} !}{\beta_{j} ! \left( \alpha_{j} - \beta_{j} \right) !}, &\qquad \beta \leq \alpha, \\
		0, &\qquad \text{otherwise}.
	\end{dcases}
\end{align*}

The heat semigroup is represented as
\begin{align*}
	\left( e^{t \Delta} \varphi \right) \left( x \right) = \left( G_{t} \ast \varphi \right) \left( x \right) = \int_{\mathbb{R}^{n}} G_{t} \left( x-y \right) \varphi \left( y \right) dy
\end{align*}
for any $\left( t, x \right) \in \left( 0, + \infty \right) \times \mathbb{R}^{n}$.
We define the dilation $\delta_{t}$ which leaves $L^{1}$-norm invariant by
\begin{align*}
	\left( \delta_{t} \varphi \right) \left( x \right) =t^{- \frac{n}{2}} \varphi \left( t^{- \frac{1}{2}} x \right), \qquad \varphi \in L_{\mathrm{loc}}^{1} \left( \mathbb{R}^{n} \right), {\ } x \in \mathbb{R}^{n}
\end{align*}
for each $t>0$.
Then, the family of the dilations $\left( \delta_{t}; t>0 \right)$ has the following properties:
\begin{itemize}
	\item[(1)]
		$\delta_{t} \delta_{s} = \delta_{ts}$ for any $t, s>0$.
	\item[(2)]
		$\left\lVert \delta_{t} \varphi \right\rVert_{q} =t^{- \frac{n}{2} \left( 1- \frac{1}{q} \right)} \left\lVert \varphi \right\rVert_{q}$ for any $t>0$, $q \in \left[ 1, + \infty \right]$ and $\varphi \in L^{q} \left( \mathbb{R}^{n} \right)$.
\end{itemize}
Moreover, by using the dilation $\delta_{t}$, we can rewrite $G_{t}$ and the derivatives of $e^{t \Delta} \varphi$ as
\begin{align*}
	G_{t} &= \delta_{t} G_{1}, \\
	\partial^{\alpha} e^{t \Delta} \varphi &= \left( \partial^{\alpha} G_{t} \right) \ast \varphi = \left( \partial^{\alpha} \left( \delta_{t} G_{1} \right) \right) \ast \varphi =t^{- \frac{\left\lvert \alpha \right\rvert}{2}} \left( \delta_{t} \left( \partial^{\alpha} G_{1} \right) \right) \ast \varphi
\end{align*}
for any $t>0$ and $\alpha \in \mathbb{Z}_{\geq 0}^{n}$.
Then, we have the following well-known estimates.

\begin{lem} \label{lem:Lp-Lq}
	Let $1 \leq q \leq p \leq + \infty$ and let $\alpha \in \mathbb{Z}_{\geq 0}^{n}$.
	Then, the estimate
	\begin{align*}
		\left\lVert \partial^{\alpha} e^{t \Delta} \varphi \right\rVert_{p} \leq t^{- \frac{n}{2} \left( \frac{1}{q} - \frac{1}{p} \right) - \frac{\left\lvert \alpha \right\rvert}{2}} \left\lVert \partial^{\alpha} G_{1} \right\rVert_{r} \left\lVert \varphi \right\rVert_{q}
	\end{align*}
	holds for any $t>0$ and $\varphi \in L^{q} \left( \mathbb{R}^{n} \right)$, where $r \in \left[ 1, + \infty \right]$ with $1/p+1=1/r+1/q$.
\end{lem}


\begin{lem} \label{lem:Lp-Lq_2}
	Let $1 \leq q \leq p \leq + \infty$.
	Then, there exists $C_{p, q} >0$ such that the estimate
	\begin{align*}
		\left\lVert e^{t \Delta} \varphi \right\rVert_{p} \leq C_{p, q} \left( 1+t \right)^{- \frac{n}{2} \left( \frac{1}{q} - \frac{1}{p} \right)} \left( \left\lVert \varphi \right\rVert_{q} + \left\lVert \varphi \right\rVert_{p} \right)
	\end{align*}
	holds for any $t \geq 0$ and $\varphi \in \left( L^{q} \cap L^{p} \right) \left( \mathbb{R}^{n} \right)$.
	In particular, if $\varphi \in \left( L^{1} \cap L^{\infty} \right) \left( \mathbb{R}^{n} \right)$, then
	\begin{align*}
		\left\lVert e^{t \Delta} \varphi \right\rVert_{p} \leq C_{p, 1} \left( 1+t \right)^{- \frac{n}{2} \left( 1- \frac{1}{p} \right)} \left( \left\lVert \varphi \right\rVert_{1} + \left\lVert \varphi \right\rVert_{\infty} \right)
	\end{align*}
	for any $t \geq 0$.
\end{lem}

Since we can prove these lemmas by simple calculations with Young's inequality, we omit the proofs.

Next, we consider asymptotic expansions of the heat semigroup.
We define the translation $\tau_{h}$ by $h \in \mathbb{R}^{n}$ as
\begin{align*}
	\left( \tau_{h} \varphi \right) \left( x \right) = \varphi \left( x-h \right), \qquad \varphi \in L_{\mathrm{loc}}^{1} \left( \mathbb{R}^{n} \right), {\ } x \in \mathbb{R}^{n}.
\end{align*}

\begin{prop} \label{pro:heat_asymptotics_0}
	Let $\varphi \in L_{1}^{1} \left( \mathbb{R}^{n} \right)$ and let $q \in \left[ 1, + \infty \right]$.
	Then, the estimate
	\begin{align*}
		t^{\frac{n}{2} \left( 1- \frac{1}{q} \right)} \left\lVert e^{t \Delta} \varphi -c_{0} \delta_{t} G_{1} \right\rVert_{q} \leq \frac{1}{2} t^{- \frac{1}{2}} \sum_{j=1}^{n} \left\lVert x_{j} G_{1} \right\rVert_{q} \left\lVert x_{j} \varphi \right\rVert_{1}
	\end{align*}
	holds for any $t>0$, where
	\begin{align*}
		c_{0} \coloneqq \int_{\mathbb{R}^{n}} \varphi \left( y \right) dy.
	\end{align*}
\end{prop}

\begin{proof}
	For any $t>0$, we have
	\begin{align*}
		\left( e^{t \Delta} \varphi -c_{0} \delta_{t} G_{1} \right) \left( x \right) &= \int_{\mathbb{R}^{n}} \left( G_{t} \left( x-y \right) -G_{t} \left( x \right) \right) \varphi \left( y \right) dy \\
		&= \int_{\mathbb{R}^{n}} \left( \int_{0}^{1} \frac{d}{d \theta} G_{t} \left( x- \theta y \right) d \theta \right) \varphi \left( y \right) dy \\
		&=- \sum_{j=1}^{n} \int_{\mathbb{R}^{n}} \int_{0}^{1} y_{j} \left( \partial_{j} G_{t} \right) \left( x- \theta y \right) \varphi \left( y \right) d \theta dy \\
		&= \frac{1}{2} t^{-1} \sum_{j=1}^{n} \int_{\mathbb{R}^{n}} \int_{0}^{1} \left( x_{j} - \theta y_{j} \right) G_{t} \left( x- \theta y \right) y_{j} \varphi \left( y \right) d \theta dy \\
		&= \frac{1}{2} t^{- \frac{n}{2} - \frac{1}{2}} \sum_{j=1}^{n} \int_{\mathbb{R}^{n}} \int_{0}^{1} \left( t^{- \frac{1}{2}} \left( x_{j} - \theta y_{j} \right) \right) G_{1} \left( t^{- \frac{1}{2}} \left( x- \theta y \right) \right) y_{j} \varphi \left( y \right) d \theta dy \\
		&= \frac{1}{2} t^{- \frac{1}{2}} \sum_{j=1}^{n} \int_{\mathbb{R}^{n}} \int_{0}^{1} \left( \tau_{\theta y} \left( \delta_{t} \left( x_{j} G_{1} \right) \right) \right) \left( x \right) y_{j} \varphi \left( y \right) d \theta dy.
	\end{align*}
	Therefore, we obtain
	\begin{align*}
		\left\lVert e^{t \Delta} \varphi -c_{0} \delta_{t} G_{1} \right\rVert_{q} &\leq \frac{1}{2} t^{- \frac{1}{2}} \sum_{j=1}^{n} \int_{\mathbb{R}^{n}} \int_{0}^{1} \left\lVert \tau_{\theta y} \left( \delta_{t} \left( x_{j} G_{1} \right) \right) \right\rVert_{q} \left\lvert y_{j} \varphi \left( y \right) \right\rvert d \theta dy \\
		&= \frac{1}{2} t^{- \frac{1}{2}} \sum_{j=1}^{n} \int_{\mathbb{R}^{n}} \int_{0}^{1} \left\lVert \delta_{t} \left( x_{j} G_{1} \right) \right\rVert_{q} \left\lvert y_{j} \varphi \left( y \right) \right\rvert d \theta dy \\
		&= \frac{1}{2} t^{- \frac{n}{2} \left( 1- \frac{1}{q} \right) - \frac{1}{2}} \sum_{j=1}^{n} \int_{\mathbb{R}^{n}} \left\lVert x_{j} G_{1} \right\rVert_{q} \left\lvert y_{j} \varphi \left( y \right) \right\rvert dy \\
		&= \frac{1}{2} t^{- \frac{n}{2} \left( 1- \frac{1}{q} \right) - \frac{1}{2}} \sum_{j=1}^{n} \left\lVert x_{j} G_{1} \right\rVert_{q} \left\lVert x_{j} \varphi \right\rVert_{1}.
	\end{align*}
\end{proof}

\begin{rem}
	For some readers, the above proof seems to be redundant.
	However, we write it to see that it is a special case of the proof of Proposition \ref{pro:heat_asymptotics_higher} below.
\end{rem}

We remark that the key point in the proof of Proposition \ref{pro:heat_asymptotics_0} is to calculate the difference $G_{t} \left( x-y \right) -G_{t} \left( x \right)$ explicitly.
Therefore, by applying Taylor's theorem, we obtain higher order asymptotic expansions of the heat semigroup.
Now, we define the Hermite polynomial of order $k$ by
\begin{align}
	\label{eq:2.1}
	H_{k} \left( x \right) = \left( -1 \right)^{k} e^{x^{2}} \left( \frac{d}{dx} \right)^{k} e^{-x^{2}}, \qquad x \in \mathbb{R}
\end{align}
for each $k \in \mathbb{Z}_{\geq 0}$.
The following representation of $H_{k}$ is well-known:
\begin{align}
	\label{eq:2.2}
	H_{k} \left( x \right) = \sum_{j=0}^{\left[ k/2 \right]} \frac{\left( -1 \right)^{j} k!}{j! \left( k-2j \right) !} \left( 2x \right)^{k-2j},
\end{align}
where $\left[ k/2 \right] \coloneqq \max \left\{ j \in \mathbb{Z}_{\geq 0}; {\,} j \leq k/2 \right\}$.
Moreover, we define the multi-variable Hermite polynomial of order $\alpha$ by $\bm{H}_{\alpha} =H_{\alpha_{1}} \otimes \cdots \otimes H_{\alpha_{n}}$, namely,
\begin{align}
	\label{eq:2.3}
	\bm{H}_{\alpha} \left( x \right) = \prod_{j=1}^{n} H_{\alpha_{j}} \left( x_{j} \right) = \left( -1 \right)^{\left\lvert \alpha \right\rvert} e^{\left\lvert x \right\rvert^{2}} \partial^{\alpha} e^{- \left\lvert x \right\rvert^{2}}, \qquad x= \left( x_{1}, \ldots, x_{n} \right) \in \mathbb{R}^{n}
\end{align}
for each $\alpha = \left( \alpha_{1}, \cdots, \alpha_{n} \right) \in \mathbb{Z}_{\geq 0}^{n}$ (see also \cite{Ozawa1990,Ozawa1991}).
Then, it follows from \eqref{eq:2.2} that
\begin{align}
	\label{eq:2.4}
	\bm{H}_{\alpha} \left( x \right) &= \prod_{j=1}^{n} \sum_{\beta_{j} =0}^{\left[ \alpha_{j} /2 \right]} \frac{\left( -1 \right)^{\beta_{j}} \alpha_{j} !}{\beta_{j} ! \left( \alpha_{j} -2 \beta_{j} \right) !} \left( 2x_{j} \right)^{\alpha_{j} -2 \beta_{j}} \nonumber \\
	&= \sum_{2 \beta \leq \alpha} \frac{\left( -1 \right)^{\left\lvert \beta \right\rvert} \alpha !}{\beta ! \left( \alpha -2 \beta \right) !} \left( 2x \right)^{\alpha -2 \beta}.
\end{align}
Using the multi-variable Hermite polynomial $\bm{H}_{\alpha}$, we can rewrite the derivatives of $G_{1}$ as
\begin{align}
	\label{eq:2.5}
	\partial^{\alpha} G_{1} \left( x \right) &= \left( 4 \pi \right)^{- \frac{n}{2}} \partial^{\alpha} \exp \left( - \left\lvert \frac{x}{2} \right\rvert^{2} \right) \nonumber \\
	&= \left( 4 \pi \right)^{- \frac{n}{2}} 2^{- \left\lvert \alpha \right\rvert} \left[ \partial_{y}^{\alpha} e^{- \left\lvert y \right\rvert^{2}} \right]_{y= \frac{x}{2}} \nonumber \\
	&= \left( 4 \pi \right)^{- \frac{n}{2}} 2^{- \left\lvert \alpha \right\rvert} \left( -1 \right)^{- \left\lvert \alpha \right\rvert} \exp \left( - \left\lvert \frac{x}{2} \right\rvert^{2} \right) \left[ \left( -1 \right)^{\left\lvert \alpha \right\rvert} e^{\left\lvert y \right\rvert^{2}} \partial_{y}^{\alpha} e^{- \left\lvert y \right\rvert^{2}} \right]_{y= \frac{x}{2}} \nonumber \\
	&= \left( -2 \right)^{- \left\lvert \alpha \right\rvert} G_{1} \left( x \right) \bm{H}_{\alpha} \left( \frac{x}{2} \right).
\end{align}
Therefore, by setting
\begin{align}
	\label{eq:2.6}
	\bm{h}_{\alpha} \left( x \right) \coloneqq \bm{H}_{\alpha} \left( \frac{x}{2} \right) = \sum_{2 \beta \leq \alpha} \frac{\left( -1 \right)^{\left\lvert \beta \right\rvert} \alpha !}{\beta ! \left( \alpha -2 \beta \right) !} x^{\alpha -2 \beta}, \qquad x \in \mathbb{R}^{n},
\end{align}
we have
\begin{align}
	\label{eq:2.7}
	\partial^{\alpha} G_{1} &= \left( -2 \right)^{- \left\lvert \alpha \right\rvert} \bm{h}_{\alpha} G_{1}, \\
	\label{eq:2.8}
	\partial^{\alpha} G_{t} &= \partial^{\alpha} \left( \delta_{t} G_{1} \right) =t^{- \frac{\left\lvert \alpha \right\rvert}{2}} \delta_{t} \left( \partial^{\alpha} G_{1} \right) = \left( -2 \right)^{- \left\lvert \alpha \right\rvert} t^{- \frac{\left\lvert \alpha \right\rvert}{2}} \delta_{t} \left( \bm{h}_{\alpha} G_{1} \right).
\end{align}
Under the above preparations, we state higher order asymptotic expansions of the heat semigroup.

\begin{prop} \label{pro:heat_asymptotics_higher}
	Let $m \in \mathbb{Z}_{\geq 0}$, $\varphi \in L_{m+1}^{1} \left( \mathbb{R}^{n} \right)$ and $q \in \left[ 1, + \infty \right]$.
	Then, the estimate
	\begin{align*}
		t^{\frac{n}{2} \left( 1- \frac{1}{q} \right)} {\,} \biggl\lVert e^{t \Delta} \varphi - \sum_{k=0}^{m} 2^{-k} t^{- \frac{k}{2}} \sum_{\left\lvert \alpha \right\rvert =k} c_{\alpha} \delta_{t} \left( \bm{h}_{\alpha} G_{1} \right) \biggr\rVert_{q} &\leq 2^{- \left( m+1 \right)} t^{- \frac{m+1}{2}} \sum_{\left\lvert \alpha \right\rvert =m+1} \frac{1}{\alpha !} \left\lVert \bm{h}_{\alpha} G_{1} \right\rVert_{q} \left\lVert x^{\alpha} \varphi \right\rVert_{1}
	\end{align*}
	holds for any $t>0$, where
	\begin{align*}
		c_{\alpha} \coloneqq \frac{1}{\alpha !} \int_{\mathbb{R}^{n}} y^{\alpha} \varphi \left( y \right) dy.
	\end{align*}
\end{prop}

\begin{rem}
	When $m=1$, the asymptotic profile of the heat semigroup is represented as
	\begin{align*}
		\sum_{k=0}^{1} 2^{-k} t^{- \frac{k}{2}} \sum_{\left\lvert \alpha \right\rvert =k} c_{\alpha} \delta_{t} \left( \bm{h}_{\alpha} G_{1} \right) =c_{0} \delta_{t} G_{1} + \frac{1}{2} t^{- \frac{1}{2}} \sum_{j=1}^{n} c_{e_{j}} \delta_{t} \left( x_{j} G_{1} \right),
	\end{align*}
	where $\left( e_{j}; j=1, \ldots, n \right)$ is the standard basis of $\mathbb{R}^{n}$.
\end{rem}

\begin{proof}[Proof of Proposition \ref{pro:heat_asymptotics_higher}]
	By Taylor's theorem, we obtain
	\begin{align*}
		G_{t} \left( x-y \right) &= \sum_{\left\lvert \alpha \right\rvert \leq m} \frac{1}{\alpha !} \left( -y \right)^{\alpha} \left( \partial^{\alpha} G_{t} \right) \left( x \right) + \sum_{\left\lvert \alpha \right\rvert =m+1} \frac{m+1}{\alpha !} \int_{0}^{1} \left( 1- \theta \right)^{m} \left( -y \right)^{\alpha} \left( \partial^{\alpha} G_{t} \right) \left( x- \theta y \right) d \theta \\
		&= \sum_{k=0}^{m} 2^{-k} t^{- \frac{k}{2}} \sum_{\left\lvert \alpha \right\rvert =k} \frac{1}{\alpha !} y^{\alpha} \left( \delta_{t} \left( \bm{h}_{\alpha} G_{1} \right) \right) \left( x \right) \\
		&\hspace{2cm} +2^{- \left( m+1 \right)} t^{- \frac{m+1}{2}} \sum_{\left\lvert \alpha \right\rvert =m+1} \frac{m+1}{\alpha !} \int_{0}^{1} \left( 1- \theta \right)^{m} y^{\alpha} \left( \tau_{\theta y} \left( \delta_{t} \left( \bm{h}_{\alpha} G_{1} \right) \right) \right) \left( x \right) d \theta,
	\end{align*}
	whence follows
	\begin{align*}
		&\biggl( e^{t \Delta} \varphi - \sum_{k=0}^{m} 2^{-k} t^{- \frac{k}{2}} \sum_{\left\lvert \alpha \right\rvert =k} c_{\alpha} \delta_{t} \left( \bm{h}_{\alpha} G_{1} \right) \biggr) \left( x \right) \\
		&\hspace{1cm} = \int_{\mathbb{R}^{n}} \biggl( G_{t} \left( x-y \right) - \sum_{k=0}^{m} 2^{-k} t^{- \frac{k}{2}} \sum_{\left\lvert \alpha \right\rvert =k} \frac{1}{\alpha !} y^{\alpha} \left( \delta_{t} \left( \bm{h}_{\alpha} G_{1} \right) \right) \left( x \right) \biggr) \varphi \left( y \right) dy \\
		&\hspace{1cm} =2^{- \left( m+1 \right)} t^{- \frac{m+1}{2}} \sum_{\left\lvert \alpha \right\rvert =m+1} \frac{m+1}{\alpha !} \int_{\mathbb{R}^{n}} \int_{0}^{1} \left( 1- \theta \right)^{m} \left( \tau_{\theta y} \left( \delta_{t} \left( \bm{h}_{\alpha} G_{1} \right) \right) \right) \left( x \right) y^{\alpha} \varphi \left( y \right) d \theta dy.
	\end{align*}
	Therefore, for any $t>0$, we have
	\begin{align*}
		&\biggl\lVert e^{t \Delta} \varphi - \sum_{k=0}^{m} 2^{-k} t^{- \frac{k}{2}} \sum_{\left\lvert \alpha \right\rvert =k} c_{\alpha} \delta_{t} \left( \bm{h}_{\alpha} G_{1} \right) \biggr\rVert_{q} \\
		&\hspace{1cm} \leq 2^{- \left( m+1 \right)} t^{- \frac{m+1}{2}} \sum_{\left\lvert \alpha \right\rvert =m+1} \frac{m+1}{\alpha !} \int_{\mathbb{R}^{n}} \int_{0}^{1} \left( 1- \theta \right)^{m} \left\lVert \tau_{\theta y} \left( \delta_{t} \left( \bm{h}_{\alpha} G_{1} \right) \right) \right\rVert_{q} \left\lvert y^{\alpha} \varphi \left( y \right) \right\rvert d \theta dy \\
		&\hspace{1cm} =2^{- \left( m+1 \right)} t^{- \frac{m+1}{2}} \sum_{\left\lvert \alpha \right\rvert =m+1} \frac{m+1}{\alpha !} \int_{\mathbb{R}^{n}} \int_{0}^{1} \left( 1- \theta \right)^{m} \left\lVert \delta_{t} \left( \bm{h}_{\alpha} G_{1} \right) \right\rVert_{q} \left\lvert y^{\alpha} \varphi \left( y \right) \right\rvert d \theta dy \\
		&\hspace{1cm} =2^{- \left( m+1 \right)} t^{- \frac{n}{2} \left( 1- \frac{1}{q} \right) - \frac{m+1}{2}} \sum_{\left\lvert \alpha \right\rvert =m+1} \frac{1}{\alpha !} \int_{\mathbb{R}^{n}} \left\lVert \bm{h}_{\alpha} G_{1} \right\rVert_{q} \left\lvert y^{\alpha} \varphi \left( y \right) \right\rvert dy \\
		&\hspace{1cm} =2^{- \left( m+1 \right)} t^{- \frac{n}{2} \left( 1- \frac{1}{q} \right) - \frac{m+1}{2}} \sum_{\left\lvert \alpha \right\rvert =m+1} \frac{1}{\alpha !} \left\lVert \bm{h}_{\alpha} G_{1} \right\rVert_{q} \left\lVert x^{\alpha} \varphi \right\rVert_{1}.
	\end{align*}
\end{proof}

\begin{rem}
	There are many methods to obtain the asymptotic expansions of the heat semigroup \cite{Duoandikoetxea-Zuazua, Fujigaki-Miyakawa, Dolbeault-Karch, Giga-Giga-Saal, Vazquez}.
	On the $m$-th order asymptotic expansion of the heat semigroup with $\varphi \in L_{m}^{1} \left( \mathbb{R}^{n} \right)$, see Appendix \ref{app:B}.
\end{rem}

As we state in the introduction, we need the following commutation relations between the heat semigroup and monomial weights in $\mathbb{R}^{n}$ to show Theorem \ref{th:P_weight}.

\begin{thm} \label{th:heat_commutator}
	Let $m \in \mathbb{Z}_{>0}$, $\varphi \in L_{m}^{1} \left( \mathbb{R}^{n} \right)$ and $\alpha \in \mathbb{Z}_{\geq 0}^{n}$ with $\left\lvert \alpha \right\rvert =m$.
	Then, $x^{\alpha} e^{t \Delta} \varphi \in L^{1} \left( \mathbb{R}^{n} \right)$ and the identity
	\begin{align}
		\label{eq:2.9}
		x^{\alpha} e^{t \Delta} \varphi -e^{t \Delta} x^{\alpha} \varphi =R_{\alpha} \left( t \right) \varphi
	\end{align}
	holds in $L^{1} \left( \mathbb{R}^{n} \right)$ for any $t>0$, where
	\begin{align}
		R_{\alpha} \left( t \right) \varphi \coloneqq \sum_{\substack{\beta + \gamma = \alpha \\ \beta \neq 0}} \frac{\alpha !}{\beta ! \gamma !} \left( -2t \partial \right)^{\beta} e^{t \Delta} x^{\gamma} \varphi + \sum_{\substack{\beta + \gamma \leq \alpha, {\ } \left\lvert \beta + \gamma \right\rvert \leq \left\lvert \alpha \right\rvert -2 \\ \left\lvert \beta \right\rvert +1 \leq \ell \leq \frac{\left\lvert \alpha \right\rvert + \left\lvert \beta \right\rvert - \left\lvert \gamma \right\rvert}{2}}} C_{\ell \beta \gamma}^{\alpha} t^{\ell} \partial^{\beta} e^{t \Delta} x^{\gamma} \varphi
	\end{align}
	for some $C_{\ell \beta \gamma}^{\alpha} \in \mathbb{R}$ independent of $t$, $x$ and $\varphi$.
\end{thm}

Theorem \ref{th:heat_commutator} yields the commutator estimates of the heat semigroup and the monomial weights.

\begin{thm} \label{th:heat_commutator_esti}
	Let $m \in \mathbb{Z}_{>0}$.
	Then, there exists $C_{m} >0$ such that the estimate
	\begin{align}
		\label{eq:2.11}
		\sum_{\left\lvert \alpha \right\rvert =m} \left\lVert x^{\alpha} e^{t \Delta} \varphi -e^{t \Delta} x^{\alpha} \varphi \right\rVert_{1} \leq C_{m} \left\{ t^{\frac{1}{2}} \bigl\lVert \left\lvert x \right\rvert^{m-1} \varphi \bigr\rVert_{1} + \left( t^{\frac{1}{2}} +t^{\frac{m}{2}} \right) \left\lVert \varphi \right\rVert_{1} \right\}
	\end{align}
	holds for any $\varphi \in L_{m}^{1} \left( \mathbb{R}^{n} \right)$ and $t>0$.
\end{thm}

By a simple calculation with the integral representation of the heat semigroup, we have
\begin{align}
	\label{eq:2.12}
	\sum_{\left\lvert \alpha \right\rvert =m} \left\lVert x^{\alpha} e^{t \Delta} \varphi \right\rVert_{1} \leq C_{m} \left( \left\lVert \left\lvert x \right\rvert^{m} \varphi \right\rVert_{1} +t^{\frac{m}{2}} \left\lVert \varphi \right\rVert_{1} \right)
\end{align}
for any $\varphi \in L_{m}^{1} \left( \mathbb{R}^{n} \right)$ and $t>0$ (cf. \cite[Lemma 2.1]{Ishige-Ishiwata-Kawakami}).
Therefore, by using \eqref{eq:2.11} instead of \eqref{eq:2.12}, we can control the heat semigroup with the weights of order $m$ by not $\left\lVert \left\lvert x \right\rvert^{m} \varphi \right\rVert_{1}$ but $\lVert \left\lvert x \right\rvert^{m-1} \varphi \rVert_{1}$.
This is crucial point to obtain weighted estimates of global solutions to \eqref{P} without the aid of the comparison principle and the iteration argument with some compactness arguments.

Here and hereafter, let $\left( e_{j}; j=1, \ldots, n \right)$ denote the standard basis of $\mathbb{R}^{n}$.
That is, for any $j \in \left\{ 1, \ldots, n \right\}$, $e_{j} \in \mathbb{R}^{n}$ is a unit vector whose components are $0$ except the $j$-th coordinate.

\begin{proof}[Proof of Theorem \ref{th:heat_commutator}]
	We regard Theorem \ref{th:heat_commutator} as the assertion with respect to $m \in \mathbb{Z}_{>0}$:
	\begin{itemize}
		\item[(A)$_{m}$]
			{\ }Let $\varphi \in L_{m}^{1} \left( \mathbb{R}^{n} \right)$ and let $\alpha \in \mathbb{Z}_{\geq 0}^{n}$ with $\left\lvert \alpha \right\rvert =m$.
			Then, $x^{\alpha} e^{t \Delta} \varphi \in L^{1} \left( \mathbb{R}^{n} \right)$ and \eqref{eq:2.9} holds \\
			{\ }for any $t>0$.
	\end{itemize}
	We show that (A)$_{m}$ is true for all $m \in \mathbb{Z}_{>0}$ by induction on $m$.
	First, we consider the case where $m=1$.
	Let $\varphi \in L_{1}^{1} \left( \mathbb{R}^{n} \right)$, $t>0$ and $\alpha \in \mathbb{Z}_{\geq 0}^{n}$ with $\left\lvert \alpha \right\rvert =1$.
	Then, there exists $j \in \left\{ 1, \ldots, n \right\}$ such that $\alpha =e_{j}$.
	Moreover, since $\varphi$, $x_{j} \varphi \in L^{1} \left( \mathbb{R}^{n} \right)$, we have $\partial_{j} e^{t \Delta} \varphi$, $e^{t \Delta} x_{j} \varphi \in L^{1} \left( \mathbb{R}^{n} \right)$ and
	\begin{align*}
		\left( e^{t \Delta} x_{j} \varphi -2t \partial_{j} e^{t \Delta} \varphi \right) \left( x \right) &= \int_{\mathbb{R}^{n}} G_{t} \left( x-y \right) y_{j} \varphi \left( y \right) dy-2t \partial_{j} \int_{\mathbb{R}^{n}} G_{t} \left( x-y \right) \varphi \left( y \right) dy \\
		&= \int_{\mathbb{R}^{n}} G_{t} \left( x-y \right) y_{j} \varphi \left( y \right) dy+ \int_{\mathbb{R}^{n}} \left( x_{j} -y_{j} \right) G_{t} \left( x-y \right) \varphi \left( y \right) dy \\
		&=x_{j} \int_{\mathbb{R}^{n}} G_{t} \left( x-y \right) \varphi \left( y \right) dy \\
		&=x_{j} \left( e^{t \Delta} \varphi \right) \left( x \right).
	\end{align*}
	Thus, we obtain $x_{j} e^{t \Delta} \varphi \in L^{1} \left( \mathbb{R}^{n} \right)$ and
	\begin{align*}
		x_{j} e^{t \Delta} \varphi -e^{t \Delta} x_{j} \varphi =-2t \partial_{j} e^{t \Delta} \varphi,
	\end{align*}
	whence follows (A)$_{1}$.

	Next, we assume that (A)$_{m}$ holds for some $m \in \mathbb{Z}_{>0}$.
	Let $\varphi \in L_{m+1}^{1} \left( \mathbb{R}^{n} \right)$, $t>0$ and $\alpha' \in \mathbb{Z}_{\geq 0}^{n}$ with $\left\lvert \alpha' \right\rvert =m+1$.
	Then, there exist $\alpha \in \mathbb{Z}_{\geq 0}^{n}$ with $\left\lvert \alpha \right\rvert =m$ and $j \in \left\{ 1, \ldots, n \right\}$ such that $\alpha' = \alpha +e_{j}$.
	Moreover, from the fact that $x^{\alpha} \varphi$, $x_{j} x^{\alpha} \varphi \in L^{1} \left( \mathbb{R}^{n} \right)$ and (A)$_{1}$, we see that $x_{j} e^{t \Delta} x^{\alpha} \varphi \in L^{1} \left( \mathbb{R}^{n} \right)$ and
	\begin{align}
		\label{eq:2.13}
		x_{j} e^{t \Delta} x^{\alpha} \varphi &=e^{t \Delta} x_{j} x^{\alpha} \varphi -2t \partial_{j} e^{t \Delta} x^{\alpha} \varphi =e^{t \Delta} x^{\alpha'} \varphi -2t \partial_{j} e^{t \Delta} x^{\alpha} \varphi.
	\end{align}
	Now, we show $x_{j} R_{\alpha} \left( t \right) \varphi \in L^{1} \left( \mathbb{R}^{n} \right)$.
	For this purpose, it suffices to prove that $x_{j} \partial^{\beta} e^{t \Delta} x^{\gamma} \varphi \in L^{1} \left( \mathbb{R}^{n} \right)$ for any $\beta, \gamma \in \mathbb{Z}_{\geq 0}^{n}$ with $\beta + \gamma \leq \alpha$ and $\left\lvert \gamma \right\rvert \leq m-1$.
	Let $\gamma \in \mathbb{Z}_{\geq 0}^{n}$ with $\left\lvert \gamma \right\rvert \leq m-1$.
	Then, from the fact that $x^{\gamma} \varphi$, $x_{j} x^{\gamma} \varphi \in L^{1} \left( \mathbb{R}^{n} \right)$ and (A)$_{1}$, it follows that $x_{j} e^{t \Delta} x^{\gamma} \varphi \in L^{1} \left( \mathbb{R}^{n} \right)$ and
	\begin{align}
		\label{eq:2.14}
		x_{j} e^{t \Delta} x^{\gamma} \varphi =e^{t \Delta} x_{j} x^{\gamma} \varphi -2t \partial_{j} e^{t \Delta} x^{\gamma} \varphi =e^{t \Delta} x^{\gamma +e_{j}} \varphi -2t \partial_{j} e^{t \Delta} x^{\gamma} \varphi.
	\end{align}
	Furthermore, since the right hand side on the above identity belongs to $W^{m, 1} \left( \mathbb{R}^{n} \right)$, for any $\beta \in \mathbb{Z}_{\geq 0}^{n}$ with $\left\lvert \beta \right\rvert \leq m$, we have $\partial^{\beta} \left( x_{j} e^{t \Delta} x^{\gamma} \varphi \right) \in L^{1} \left( \mathbb{R}^{n} \right)$ and
	\begin{align*}
		\partial^{\beta} \left( x_{j} e^{t \Delta} x^{\gamma} \varphi \right) &= \partial^{\beta} \left( e^{t \Delta} x^{\gamma +e_{j}} \varphi -2t \partial_{j} e^{t \Delta} x^{\gamma} \varphi \right) \\
		&= \partial^{\beta} e^{t \Delta} x^{\gamma +e_{j}} \varphi -2t \partial^{\beta +e_{j}} e^{t \Delta} x^{\gamma} \varphi.
	\end{align*}
	On the other hand, by a simple calculation, we obtain
	\begin{align}
		\label{eq:2.15}
		\partial^{\beta} \left( x_{j} e^{t \Delta} x^{\gamma} \varphi \right) = \begin{dcases}
			x_{j} \partial^{\beta} e^{t \Delta} x^{\gamma} \varphi + \beta_{j} \partial^{\beta -e_{j}} e^{t \Delta} x^{\gamma} \varphi &\quad \text{if} \quad \beta_{j} \geq 1, \\
			x_{j} \partial^{\beta} e^{t \Delta} x^{\gamma} \varphi &\quad \text{if} \quad \beta_{j} =0,
		\end{dcases}
	\end{align}
	where $\beta_{j}$ is the $j$-th component of $\beta$.
	In any case, we conclude that $x_{j} \partial^{\beta} e^{t \Delta} x^{\gamma} \varphi \in L^{1} \left( \mathbb{R}^{n} \right)$, whence follows $x_{j} R_{\alpha} \left( t \right) \varphi \in L^{1} \left( \mathbb{R}^{n} \right)$.
	Therefore, from the fact that $x_{j} e^{t \Delta} x^{\alpha} \varphi$, $x_{j} R_{\alpha} \left( t \right) \varphi \in L^{1} \left( \mathbb{R}^{n} \right)$, (A)$_{m}$, and \eqref{eq:2.13}, it follows that $x^{\alpha'} e^{t \Delta} \varphi \in L^{1} \left( \mathbb{R}^{n} \right)$ and
	\begin{align*}
		x^{\alpha'} e^{t \Delta} \varphi &=x_{j} x^{\alpha} e^{t \Delta} \varphi \\
		&=x_{j} e^{t \Delta} x^{\alpha} \varphi +x_{j} R_{\alpha} \left( t \right) \varphi \\
		&=e^{t \Delta} x^{\alpha'} \varphi -2t \partial_{j} e^{t \Delta} x^{\alpha} \varphi +x_{j} R_{\alpha} \left( t \right) \varphi.
	\end{align*}
	Finally, to complete the proof, we show that the terms $-2t \partial_{j} e^{t \Delta} x^{\alpha} \varphi$ and $x_{j} R_{\alpha} \left( t \right) \varphi$ in the above identity are represented as
	\begin{align*}
		R_{\alpha'} \left( t \right) \varphi = \sum_{\substack{\beta' + \gamma' = \alpha' \\ \beta' \neq 0}} \frac{\alpha' !}{\beta' ! \gamma' !} \left( -2t \partial \right)^{\beta'} e^{t \Delta} x^{\gamma'} \varphi + \sum_{\substack{\beta' + \gamma' \leq \alpha', {\ } \left\lvert \beta' + \gamma' \right\rvert \leq \left\lvert \alpha' \right\rvert -2 \\ \left\lvert \beta' \right\rvert +1 \leq \ell' \leq \frac{\lvert \alpha' \rvert + \lvert \beta' \rvert - \lvert \gamma' \rvert}{2}}} C_{\ell' \beta' \gamma'}^{\alpha'} t^{\ell'} \partial^{\beta'} e^{t \Delta} x^{\gamma'} \varphi.
	\end{align*}
	It is clear that the term $-2t \partial_{j} e^{t \Delta} x^{\alpha} \varphi$ is a part of the first sum in $R_{\alpha'} \left( t \right) \varphi$ with $\left( \beta', \gamma' \right) = \left( e_{j}, \alpha \right)$.
	Let $\beta, \gamma \in \mathbb{Z}_{\geq 0}^{n}$ with $\beta + \gamma = \alpha$ and $\beta \neq 0$.
	If $\beta_{j} \geq 1$, then it follows from \eqref{eq:2.14} and \eqref{eq:2.15} that
	\begin{align}
		\label{eq:2.16}
		x_{j} \left( -2t \partial \right)^{\beta} e^{t \Delta} x^{\gamma} \varphi &= \left( -2t \partial \right)^{\beta} \left( x_{j} e^{t \Delta} x^{\gamma} \varphi \right) - \beta_{j} \left( -2t \right)^{\left\lvert \beta \right\rvert} \partial^{\beta -e_{j}} e^{t \Delta} x^{\gamma} \varphi \nonumber \\
		&= \left( -2t \partial \right)^{\beta} \left( e^{t \Delta} x^{\gamma +e_{j}} \varphi -2t \partial_{j} e^{t \Delta} x^{\gamma} \varphi \right) - \beta_{j} \left( -2 \right)^{\left\lvert \beta \right\rvert} t^{\left\lvert \beta \right\rvert} \partial^{\beta -e_{j}} e^{t \Delta} x^{\gamma} \varphi \nonumber \\
		&= \left( -2t \partial \right)^{\beta} e^{t \Delta} x^{\gamma +e_{j}} \varphi + \left( -2t \partial \right)^{\beta +e_{j}} e^{t \Delta} x^{\gamma} \varphi - \beta_{j} \left( -2 \right)^{\left\lvert \beta \right\rvert} t^{\left\lvert \beta \right\rvert} \partial^{\beta -e_{j}} e^{t \Delta} x^{\gamma} \varphi.
	\end{align}
	Similarly, if $\beta_{j} =0$, then we have
	\begin{align}
		\label{eq:2.17}
		x_{j} \left( -2t \partial \right)^{\beta} e^{t \Delta} x^{\gamma} \varphi &= \left( -2t \partial \right)^{\beta} \left( x_{j} e^{t \Delta} x^{\gamma} \varphi \right) \nonumber \\
		&= \left( -2t \partial \right)^{\beta} e^{t \Delta} x^{\gamma +e_{j}} \varphi + \left( -2t \partial \right)^{\beta +e_{j}} e^{t \Delta} x^{\gamma} \varphi.
	\end{align}
	The terms $\left( -2t \partial \right)^{\beta} e^{t \Delta} x^{\gamma +e_{j}} \varphi$ and $\left( -2t \partial \right)^{\beta +e_{j}} e^{t \Delta} x^{\gamma} \varphi$ in \eqref{eq:2.16} or \eqref{eq:2.17} are parts of the first sum in $R_{\alpha'} \left( t \right) \varphi$ with $\left( \beta', \gamma' \right) = \left( \beta, \gamma +e_{j} \right)$, $\left( \beta +e_{j}, \gamma \right)$, respectively.
	In addition, since
	\begin{align*}
		&-2t \partial_{j} e^{t \Delta} x^{\alpha} \varphi + \sum_{\substack{\beta + \gamma = \alpha \\ \beta \neq 0}} \frac{\alpha !}{\beta ! \gamma !} \left( \left( -2t \partial \right)^{\beta} e^{t \Delta} x^{\gamma +e_{j}} \varphi + \left( -2t \partial \right)^{\beta +e_{j}} e^{t \Delta} x^{\gamma} \varphi \right) \\
		&\hspace{1cm} = \sum_{\substack{\beta + \gamma = \alpha \\ \beta \neq 0}} \frac{\alpha !}{\beta ! \gamma !} \left( -2t \partial \right)^{\beta} e^{t \Delta} x^{\gamma +e_{j}} \varphi + \sum_{\beta + \gamma = \alpha} \frac{\alpha !}{\beta ! \gamma !} \left( -2t \partial \right)^{\beta +e_{j}} e^{t \Delta} x^{\gamma} \varphi \\
		&\hspace{1cm} = \sum_{\substack{\beta' + \gamma' = \alpha' \\ 0 \neq \beta' \leq \alpha}} \frac{\alpha !}{\beta' ! \left( \gamma' -e_{j} \right) !} \left( -2t \partial \right)^{\beta'} e^{t \Delta} x^{\gamma'} \varphi + \sum_{\substack{\beta' + \gamma' = \alpha' \\ \gamma' \leq \alpha}} \frac{\alpha !}{\left( \beta' -e_{j} \right) ! \gamma' !} \left( -2t \partial \right)^{\beta'} e^{t \Delta} x^{\gamma'} \varphi \\
		&\hspace{1cm} = \sum_{\substack{\beta' + \gamma' = \alpha' \\ \beta' \neq 0}} \left( \binom{\alpha}{\beta'} + \binom{\alpha}{\gamma'} \right) \left( -2t \partial \right)^{\beta'} e^{t \Delta} x^{\gamma'} \varphi \\
		&\hspace{1cm} = \sum_{\substack{\beta' + \gamma' = \alpha' \\ \beta' \neq 0}} \binom{\alpha'}{\beta'} \left( -2t \partial \right)^{\beta'} e^{t \Delta} x^{\gamma'} \varphi \\
		&\hspace{1cm} = \sum_{\substack{\beta' + \gamma' = \alpha' \\ \beta' \neq 0}} \frac{\alpha' !}{\beta' ! \gamma' !} \left( -2t \partial \right)^{\beta'} e^{t \Delta} x^{\gamma'} \varphi,
	\end{align*}
	all components of the first sum in $R_{\alpha'} \left( t \right) \varphi$ have appeared.
	On the other hand, by taking $\left( \ell', \beta', \gamma' \right) = \left( \left\lvert \beta \right\rvert, \beta -e_{j}, \gamma \right)$ with $e_{j} \leq \beta$, we have
	\begin{itemize}
		\item
			$\beta ' + \gamma' = \beta -e_{j} + \gamma = \alpha -e_{j} = \alpha' -2e_{j} \leq \alpha'$,
		\item
			$\left\lvert \beta' + \gamma' \right\rvert = \left\lvert \beta -e_{j} + \gamma \right\rvert = \left\lvert \alpha \right\rvert -1= \left\lvert \alpha' \right\rvert -2$,
		\item
			$\left\lvert \beta' \right\rvert +1= \left\lvert \beta -e_{j} \right\rvert +1= \left\lvert \beta \right\rvert = \ell'$,
		\item
			$\ell' \leq \dfrac{\left\lvert \alpha' \right\rvert + \left\lvert \beta' \right\rvert - \left\lvert \gamma' \right\rvert}{2} {\ } \Leftrightarrow {\ } \left\lvert \beta \right\rvert \leq \dfrac{\left( \left\lvert \alpha \right\rvert +1 \right) + \left( \left\lvert \beta \right\rvert -1 \right) - \left\lvert \gamma \right\rvert}{2} = \dfrac{\left\lvert \alpha \right\rvert + \left\lvert \beta \right\rvert - \left\lvert \gamma \right\rvert}{2} = \left\lvert \beta \right\rvert$.
	\end{itemize}
	This implies that the term $- \beta_{j} \left( -2 \right)^{\left\lvert \beta \right\rvert} t^{\left\lvert \beta \right\rvert} \partial^{\beta -e_{j}} e^{t \Delta} x^{\gamma} \varphi$ in \eqref{eq:2.16} is a part of the second sum in $R_{\alpha'} \left( t \right) \varphi$ with $\left( \ell', \beta', \gamma' \right) = \left( \left\lvert \beta \right\rvert, \beta -e_{j}, \gamma \right)$.
	Therefore, we conclude that the components of the first sum in $x_{j} R_{\alpha} \left( t \right) \varphi$ are represented as parts of $R_{\alpha'} \left( t \right) \varphi$.
	Next, we take $\beta, \gamma \in \mathbb{Z}_{\geq 0}^{n}$ and $\ell \in \mathbb{Z}_{\geq 0}$ satisfying
	\begin{align*}
		\beta + \gamma \leq \alpha, \qquad \left\lvert \beta + \gamma \right\rvert \leq \left\lvert \alpha \right\rvert -2, \qquad \left\lvert \beta \right\rvert +1 \leq \ell \leq \frac{\left\lvert \alpha \right\rvert + \left\lvert \beta \right\rvert - \left\lvert \gamma \right\rvert}{2}.
	\end{align*}
	If $\beta_{j} \geq 1$, then it follows from \eqref{eq:2.14} and \eqref{eq:2.15} that
	\begin{align}
		\label{eq:2.18}
		x_{j} t^{\ell} \partial^{\beta} e^{t \Delta} x^{\gamma} \varphi &=t^{\ell} \partial^{\beta} \left( x_{j} e^{t \Delta} x^{\gamma} \varphi \right) - \beta_{j} t^{\ell} \partial^{\beta -e_{j}} e^{t \Delta} x^{\gamma} \varphi \nonumber \\
		&=t^{\ell} \partial^{\beta} \left( e^{t \Delta} x^{\gamma +e_{j}} \varphi -2t \partial_{j} e^{t \Delta} x^{\gamma} \varphi \right) - \beta_{j} t^{\ell} \partial^{\beta -e_{j}} e^{t \Delta} x^{\gamma} \varphi \nonumber \\
		&=t^{\ell} \partial^{\beta} e^{t \Delta} x^{\gamma +e_{j}} \varphi -2t^{\ell +1} \partial^{\beta +e_{j}} e^{t \Delta} x^{\gamma} \varphi - \beta_{j} t^{\ell} \partial^{\beta -e_{j}} e^{t \Delta} x^{\gamma} \varphi.
	\end{align}
	In the same way, if $\beta_{j} =0$, then we have
	\begin{align}
		\label{eq:2.19}
		x_{j} t^{\ell} \partial^{\beta} e^{t \Delta} x^{\gamma} \varphi &=t^{\ell} \partial^{\beta} \left( x_{j} e^{t \Delta} x^{\gamma} \varphi \right) \nonumber \\
		&=t^{\ell} \partial^{\beta} e^{t \Delta} x^{\gamma +e_{j}} \varphi -2t^{\ell +1} \partial^{\beta +e_{j}} e^{t \Delta} x^{\gamma} \varphi.
	\end{align}
	Taking $\left( \ell', \beta', \gamma' \right) = \left( \ell, \beta, \gamma +e_{j} \right)$, we have
	\begin{itemize}
		\item
			$\beta ' + \gamma' = \beta + \gamma +e_{j} \leq \alpha +e_{j} = \alpha'$,
		\item
			$\left\lvert \beta' + \gamma' \right\rvert = \left\lvert \beta + \gamma +e_{j} \right\rvert = \left\lvert \beta + \gamma \right\rvert +1 \leq \left\lvert \alpha \right\rvert -1= \left\lvert \alpha' \right\rvert -2$,
		\item
			$\left\lvert \beta' \right\rvert +1= \left\lvert \beta \right\rvert +1 \leq \ell = \ell'$,
		\item
			$\ell' \leq \dfrac{\left\lvert \alpha' \right\rvert + \left\lvert \beta' \right\rvert - \left\lvert \gamma' \right\rvert}{2} {\ } \Leftrightarrow {\ } \ell \leq \dfrac{\left( \left\lvert \alpha \right\rvert +1 \right) + \left\lvert \beta \right\rvert - \left( \left\lvert \gamma \right\rvert +1 \right)}{2} = \dfrac{\left\lvert \alpha \right\rvert + \left\lvert \beta \right\rvert - \left\lvert \gamma \right\rvert}{2}$.
	\end{itemize}
	Hence, the term $t^{\ell} \partial^{\beta} e^{t \Delta} x^{\gamma +e_{j}} \varphi$ in \eqref{eq:2.18} or \eqref{eq:2.19} is a part of the second sum in $R_{\alpha'} \left( t \right) \varphi$ with $\left( \ell', \beta', \gamma' \right) = \left( \ell, \beta, \gamma +e_{j} \right)$.
	Next, taking $\left( \ell', \beta', \gamma' \right) = \left( \ell +1, \beta +e_{j}, \gamma \right)$, we obtain
	\begin{itemize}
		\item
			$\beta ' + \gamma' = \beta +e_{j} + \gamma \leq \alpha +e_{j} = \alpha'$,
		\item
			$\left\lvert \beta' + \gamma' \right\rvert = \left\lvert \beta +e_{j} + \gamma \right\rvert = \left\lvert \beta + \gamma \right\rvert +1 \leq \left\lvert \alpha \right\rvert -1= \left\lvert \alpha' \right\rvert -2$,
		\item
			$\left\lvert \beta' \right\rvert +1= \left\lvert \beta +e_{j} \right\rvert +1= \left( \left\lvert \beta \right\rvert +1 \right) +1 \leq \ell +1= \ell'$,
		\item
			$\ell' \leq \dfrac{\left\lvert \alpha' \right\rvert + \left\lvert \beta' \right\rvert - \left\lvert \gamma' \right\rvert}{2} {\ } \Leftrightarrow {\ } \ell \leq \dfrac{\left( \left\lvert \alpha \right\rvert +1 \right) + \left( \left\lvert \beta \right\rvert +1 \right) - \left\lvert \gamma \right\rvert}{2} -1= \dfrac{\left\lvert \alpha \right\rvert + \left\lvert \beta \right\rvert - \left\lvert \gamma \right\rvert}{2}$.
	\end{itemize}
	This implies that the term $-2t^{\ell +1} \partial^{\beta +e_{j}} e^{t \Delta} x^{\gamma} \varphi$ in \eqref{eq:2.18} or \eqref{eq:2.19} is a part of the second sum in $R_{\alpha'} \left( t \right) \varphi$ with $\left( \ell', \beta', \gamma' \right) = \left( \ell +1, \beta +e_{j}, \gamma \right)$.
	Finally, taking $\left( \ell', \beta', \gamma' \right) = \left( \ell, \beta -e_{j}, \gamma \right)$ with $e_{j} \leq \beta$, we have
	\begin{itemize}
		\item
			$\beta ' + \gamma' = \beta -e_{j} + \gamma \leq \alpha -e_{j} = \alpha' -2e_{j} \leq \alpha'$,
		\item
			$\left\lvert \beta' + \gamma' \right\rvert = \left\lvert \beta -e_{j} + \gamma \right\rvert = \left\lvert \beta + \gamma \right\rvert -1 \leq \left\lvert \alpha \right\rvert -3= \left\lvert \alpha' \right\rvert -4 \leq \left\lvert \alpha' \right\rvert -2$,
		\item
			$\left\lvert \beta' \right\rvert +1= \left\lvert \beta -e_{j} \right\rvert +1 = \left\lvert \beta \right\rvert \leq \ell -1 \leq \ell = \ell'$,
		\item
			$\ell' \leq \dfrac{\left\lvert \alpha' \right\rvert + \left\lvert \beta' \right\rvert - \left\lvert \gamma' \right\rvert}{2} {\ } \Leftrightarrow {\ } \ell \leq \dfrac{\left( \left\lvert \alpha \right\rvert +1 \right) + \left( \left\lvert \beta \right\rvert -1 \right) - \left\lvert \gamma \right\rvert}{2} = \dfrac{\left\lvert \alpha \right\rvert + \left\lvert \beta \right\rvert - \left\lvert \gamma \right\rvert}{2}$.
	\end{itemize}
	Thus, the term $- \beta_{j} t^{\ell} \partial^{\beta -e_{j}} e^{t \Delta} x^{\gamma} \varphi$ in \eqref{eq:2.18} is a part of the second sum in $R_{\alpha'} \left( t \right) \varphi$ with $\left( \ell', \beta', \gamma' \right) = \left( \ell, \beta -e_{j}, \gamma \right)$, whence follows that each component of the second sum in $x_{j} R_{\alpha} \left( t \right) \varphi$ is represented as a component of the second sum in $R_{\alpha'} \left( t \right) \varphi$.
	This completes the proof.
\end{proof}

\begin{proof}[Proof of Theorem \ref{th:heat_commutator_esti}]
	By virtue of Theorem \ref{th:heat_commutator}, it is sufficient to estimate $R_{\alpha} \left( t \right) \varphi$ for any $t>0$.
	For the case where $m=1$, it follows from Lemma \ref{lem:Lp-Lq} that
	\begin{align*}
		\sum_{\left\lvert \alpha \right\rvert =1} \left\lVert R_{\alpha} \left( t \right) \varphi \right\rVert_{1} &= \sum_{j=1}^{n} \left\lVert -2t \partial_{j} e^{t \Delta} \varphi \right\rVert_{1} \\
		&\leq 2t^{\frac{1}{2}} \sum_{j=1}^{n} \left\lVert \partial_{j} G_{1} \right\rVert_{1} \left\lVert \varphi \right\rVert_{1} \\
		&=2t^{\frac{1}{2}} \left\lVert \nabla G_{1} \right\rVert_{1} \left\lVert \varphi \right\rVert_{1}.
	\end{align*}
	Next, we consider the case where $m \geq 2$.
	Here and hereafter, different positive constants independent of $t$ are denoted by the same letter $C$.
	By Lemma \ref{lem:Lp-Lq}, for any $\alpha \in \mathbb{Z}_{\geq 0}^{n}$ with $\left\lvert \alpha \right\rvert =m$, we have
	\begin{align*}
		\left\lVert R_{\alpha} \left( t \right) \varphi \right\rVert_{1} &\leq C \sum_{\substack{\beta + \gamma = \alpha \\ \beta \neq 0}} t^{\left\lvert \beta \right\rvert} \bigl\lVert \partial^{\beta} e^{t \Delta} x^{\gamma} \varphi \bigr\rVert_{1} +C \sum_{\substack{\beta + \gamma \leq \alpha, {\ } \left\lvert \beta + \gamma \right\rvert \leq m-2 \\ \left\lvert \beta \right\rvert +1 \leq \ell \leq \frac{m+ \left\lvert \beta \right\rvert - \left\lvert \gamma \right\rvert}{2}}} t^{\ell} \bigl\lVert \partial^{\beta} e^{t \Delta} x^{\gamma} \varphi \bigr\rVert_{1} \\
		&\leq C \sum_{\substack{\left\lvert \beta \right\rvert + \left\lvert \gamma \right\rvert =m \\ \left\lvert \beta \right\rvert \geq 1}} t^{\frac{\left\lvert \beta \right\rvert}{2}} \left\lVert x^{\gamma} \varphi \right\rVert_{1} +C \sum_{\substack{\left\lvert \beta \right\rvert + \left\lvert \gamma \right\rvert \leq m-2 \\ \left\lvert \beta \right\rvert +1 \leq \ell \leq \frac{m+ \left\lvert \beta \right\rvert - \left\lvert \gamma \right\rvert}{2}}} t^{\ell - \frac{\left\lvert \beta \right\rvert}{2}} \left\lVert x^{\gamma} \varphi \right\rVert_{1},
	\end{align*}
	whence follows
	\begin{align}
		\label{eq:2.20}
		\sum_{\left\lvert \alpha \right\rvert =m} \left\lVert R_{\alpha} \left( t \right) \varphi \right\rVert_{1} &\leq C \sum_{\substack{\left\lvert \beta \right\rvert + \left\lvert \gamma \right\rvert =m \\ \left\lvert \beta \right\rvert \geq 1}} t^{\frac{\left\lvert \beta \right\rvert}{2}} \left\lVert x^{\gamma} \varphi \right\rVert_{1} +C \sum_{j=0}^{m-2} \sum_{\left\lvert \beta \right\rvert + \left\lvert \gamma \right\rvert =j} \sum_{\left\lvert \beta \right\rvert +1 \leq \ell \leq \frac{m+ \left\lvert \beta \right\rvert - \left\lvert \gamma \right\rvert}{2}} t^{\ell - \frac{\left\lvert \beta \right\rvert}{2}} \left\lVert x^{\gamma} \varphi \right\rVert_{1}.
	\end{align}
	Now, let $\beta, \gamma \in \mathbb{Z}_{\geq 0}^{n}$ with $\left\lvert \beta \right\rvert + \left\lvert \gamma \right\rvert =m$ and $\left\lvert \beta \right\rvert \geq 1$.
	Then,
	\begin{align*}
		\frac{\left\lvert \beta \right\rvert -1}{m-1} + \frac{\left\lvert \gamma \right\rvert}{m-1} =1, \qquad 0 \leq \frac{\left\lvert \beta \right\rvert -1}{m-1} \leq 1, \qquad 0 \leq \frac{\left\lvert \gamma \right\rvert}{m-1} \leq 1.
	\end{align*}
	Therefore, H\"{o}lder's inequality implies
	\begin{align}
		\label{eq:2.21}
		t^{\frac{\left\lvert \beta \right\rvert}{2}} \left\lVert x^{\gamma} \varphi \right\rVert_{1} &\leq t^{\frac{\left\lvert \beta \right\rvert}{2}} \bigl\lVert \left\lvert x \right\rvert^{\left\lvert \gamma \right\rvert} \varphi \bigr\rVert_{1} \nonumber \\
		&\leq t^{\frac{\left\lvert \beta \right\rvert}{2}} \left\lVert \varphi \right\rVert_{1}^{\frac{\left\lvert \beta \right\rvert -1}{m-1}} \bigl\lVert \left\lvert x \right\rvert^{m-1} \varphi \bigr\rVert_{1}^{\frac{\left\lvert \gamma \right\rvert}{m-1}} \nonumber \\
		&=t^{\frac{1}{2}} \left( t^{\frac{m-1}{2}} \left\lVert \varphi \right\rVert_{1} \right)^{\frac{\left\lvert \beta \right\rvert -1}{m-1}} \bigl\lVert \left\lvert x \right\rvert^{m-1} \varphi \bigr\rVert_{1}^{\frac{\left\lvert \gamma \right\rvert}{m-1}} \nonumber \\
		&\leq t^{\frac{1}{2}} \left( t^{\frac{m-1}{2}} \left\lVert \varphi \right\rVert_{1} + \bigl\lVert \left\lvert x \right\rvert^{m-1} \varphi \bigr\rVert_{1} \right) \nonumber \\
		&=t^{\frac{m}{2}} \left\lVert \varphi \right\rVert_{1} +t^{\frac{1}{2}} \bigl\lVert \left\lvert x \right\rvert^{m-1} \varphi \bigr\rVert_{1}.
	\end{align}
	Similarly, we take $j, \ell \in \mathbb{Z}_{\geq 0}$ and $\beta, \gamma \in \mathbb{Z}_{\geq 0}^{n}$ satisfying
	\begin{align*}
		0 \leq j \leq m-2, \qquad \left\lvert \beta \right\rvert + \left\lvert \gamma \right\rvert =j, \qquad \left\lvert \beta \right\rvert +1 \leq \ell \leq \frac{m+ \left\lvert \beta \right\rvert - \left\lvert \gamma \right\rvert}{2}.
	\end{align*}
	Then,
	\begin{gather*}
		\frac{m+ \left\lvert \beta \right\rvert -j-1}{m-1} + \frac{\left\lvert \gamma \right\rvert}{m-1} =1, \qquad 0 \leq \frac{m+ \left\lvert \beta \right\rvert -j-1}{m-1} \leq 1, \qquad 0 \leq \frac{\left\lvert \gamma \right\rvert}{m-1} \leq 1, \\
		\left\lvert \beta \right\rvert +1 \leq 2 \ell - \left\lvert \beta \right\rvert -1 \leq m+ \left\lvert \beta \right\rvert - \left( \left\lvert \gamma \right\rvert + \left\lvert \beta \right\rvert \right) -1=m+ \left\lvert \beta \right\rvert -j-1.
	\end{gather*}
	Hence, we obtain
	\begin{align}
		\label{eq:2.22}
		t^{\ell - \frac{\left\lvert \beta \right\rvert}{2}} \left\lVert x^{\gamma} \varphi \right\rVert_{1} &\leq t^{\ell - \frac{\left\lvert \beta \right\rvert}{2}} \bigl\lVert \left\lvert x \right\rvert^{\left\lvert \gamma \right\rvert} \varphi \bigr\rVert_{1} \nonumber \\
		&\leq t^{\frac{2 \ell - \left\lvert \beta \right\rvert}{2}} \left\lVert \varphi \right\rVert_{1}^{\frac{m+ \left\lvert \beta \right\rvert -j-1}{m-1}} \bigl\lVert \left\lvert x \right\rvert^{m-1} \varphi \bigr\rVert_{1}^{\frac{\left\lvert \gamma \right\rvert}{m-1}} \nonumber \\
		&=t^{\frac{1}{2}} \left( t^{\frac{\left( m-1 \right) \left( 2 \ell - \left\lvert \beta \right\rvert -1 \right)}{2 \left( m+ \left\lvert \beta \right\rvert -j-1 \right)}} \left\lVert \varphi \right\rVert_{1} \right)^{\frac{m+ \left\lvert \beta \right\rvert -j-1}{m-1}} \bigl\lVert \left\lvert x \right\rvert^{m-1} \varphi \bigr\rVert_{1}^{\frac{\left\lvert \gamma \right\rvert}{m-1}} \nonumber \\
		&\leq t^{\frac{1}{2}} \left( t^{\frac{\left( m-1 \right) \left( 2 \ell - \left\lvert \beta \right\rvert -1 \right)}{2 \left( m+ \left\lvert \beta \right\rvert -j-1 \right)}} \left\lVert \varphi \right\rVert_{1} + \bigl\lVert \left\lvert x \right\rvert^{m-1} \varphi \bigr\rVert_{1} \right) \nonumber \\
		&\leq \left( t^{\frac{1}{2}} + t^{\frac{m}{2}} \right) \left\lVert \varphi \right\rVert_{1} +t^{\frac{1}{2}} \bigl\lVert \left\lvert x \right\rvert^{m-1} \varphi \bigr\rVert_{1}.
	\end{align}
	Here, we have used the inequalities
	\begin{align*}
		t^{\frac{\left( m-1 \right) \left( 2 \ell - \left\lvert \beta \right\rvert -1 \right)}{2 \left( m+ \left\lvert \beta \right\rvert -j-1 \right)}} \leq \begin{dcases}
			1, &\quad 0<t \leq 1, \\
			t^{\frac{m-1}{2}}, &\quad t>1,
		\end{dcases}
	\end{align*}
	which follow from the following relation of exponents:
	\begin{align*}
		0 \leq \frac{\left\lvert \beta \right\rvert +1}{m+ \left\lvert \beta \right\rvert -j-1} \leq \frac{2 \ell - \left\lvert \beta \right\rvert -1}{m+ \left\lvert \beta \right\rvert -j-1} \leq 1.
	\end{align*}
	Finally, combining \eqref{eq:2.20}, \eqref{eq:2.21}, and \eqref{eq:2.22}, we arrive at the desired estimate.
\end{proof}

\begin{rem} \label{rem:heat_remainder_esti}
	From the proofs of Theorems \ref{th:heat_commutator} and \ref{th:heat_commutator_esti}, we see that if $m \in \mathbb{Z}_{>0}$ and $\varphi \in L^{1}_{m} \left( \mathbb{R}^{n} \right)$, then for any $\alpha \in \mathbb{Z}_{\geq 0}^{n}$ with $\left\lvert \alpha \right\rvert =m$, $j \in \left\{ 1, \ldots, n \right\}$ and $t>0$, $x_{j} R_{\alpha} \left( t \right) \varphi$ is represented as a part of $R_{\alpha +e_{j}} \left( t \right) \varphi$ and the estimate
	\begin{align*}
		\sum_{j=1}^{n} \sum_{\left\lvert \alpha \right\rvert =m} \left\lVert x_{j} R_{\alpha} \left( t \right) \varphi \right\rVert_{1} \leq C_{m+1} \left\{ t^{\frac{1}{2}} \left\lVert \left\lvert x \right\rvert^{m} \varphi \right\rVert_{1} + \left( t^{\frac{1}{2}} +t^{\frac{m+1}{2}} \right) \left\lVert \varphi \right\rVert_{1} \right\}
	\end{align*}
	holds for some $C_{m+1} >0$ independent of $t$ and $\varphi$.
\end{rem}

The following lemma will be used in the proof of Theorem \ref{th:P_weight} to calculate weighted estimates with approximation.

\begin{lem} \label{lem:heat_weight_appro}
	Let $w \in W^{2, \infty} \left( \mathbb{R}^{n} \right)$ and let $\varphi \in L^{1} \left( \mathbb{R}^{n} \right)$.
	Then, the estimate
	\begin{align}
		\left\lVert we^{t \Delta} \varphi -e^{t \Delta} w \varphi \right\rVert_{1} &\leq \left( \left\lVert \Delta w \right\rVert_{\infty} t+ \left\lVert \nabla w \right\rVert_{\infty} \left\lVert \nabla G_{1} \right\rVert_{1} t^{\frac{1}{2}} \right) \left\lVert \varphi \right\rVert_{1}
	\end{align}
	holds for any $t>0$.
\end{lem}

\begin{proof}
	From the identity
	\begin{align*}
		we^{t \Delta} \varphi -e^{t \Delta} w \varphi &= \int_{0}^{t} \frac{d}{ds} \left( e^{\left( t-s \right) \Delta} we^{s \Delta} \varphi \right) ds \\
		&= \int_{0}^{t} e^{\left( t-s \right) \Delta} \left( - \Delta \left( we^{s \Delta} \varphi \right) +w \Delta e^{s \Delta} \varphi \right) ds \\
		&= \int_{0}^{t} e^{\left( t-s \right) \Delta} \left( - \Delta we^{s \Delta} \varphi -2 \nabla w \cdot \nabla e^{s \Delta} \varphi \right) ds
	\end{align*}
	and Lemma \ref{lem:Lp-Lq}, we have
	\begin{align*}
		\left\lVert we^{t \Delta} \varphi -e^{t \Delta} w \varphi \right\rVert_{1} &\leq \int_{0}^{t} \bigl\lVert e^{\left( t-s \right) \Delta} \left( - \Delta we^{s \Delta} \varphi -2 \nabla w \cdot \nabla e^{s \Delta} \varphi \right) \bigr\rVert_{1} ds \\
		&\leq \left\lVert G_{1} \right\rVert_{1} \int_{0}^{t} \left\lVert - \Delta we^{s \Delta} \varphi -2 \nabla w \cdot \nabla e^{s \Delta} \varphi \right\rVert_{1} ds \\
		&\leq \left\lVert \Delta w \right\rVert_{\infty} \int_{0}^{t} \left\lVert e^{s \Delta} \varphi \right\rVert_{1} ds+2 \left\lVert \nabla w \right\rVert_{\infty} \int_{0}^{t} \left\lVert \nabla e^{s \Delta} \varphi \right\rVert_{1} ds \\
		&\leq \left\lVert \Delta w \right\rVert_{\infty} \left\lVert G_{1} \right\rVert_{1} \left\lVert \varphi \right\rVert_{1} \int_{0}^{t} ds+2 \left\lVert \nabla w \right\rVert_{\infty} \left\lVert \nabla G_{1} \right\rVert_{1} \left\lVert \varphi \right\rVert_{1} \int_{0}^{t} s^{- \frac{1}{2}} ds \\
		&= \left( \left\lVert \Delta w \right\rVert_{\infty} t+ \left\lVert \nabla w \right\rVert_{\infty} \left\lVert \nabla G_{1} \right\rVert_{1} t^{\frac{1}{2}} \right) \left\lVert \varphi \right\rVert_{1}.
	\end{align*}
\end{proof}

\section{Proof of Theorem \ref{th:P_weight}} \label{sec:P_weight}

	We show Theorem \ref{th:P_weight} by induction on $m \in \mathbb{Z}_{>0}$.
	First of all, we introduce approximate functions of the monomial weights.
	For $j \in \left\{ 1, \ldots, n \right\}$ and $\varepsilon \in \left( 0, 1 \right]$, we define a function $w_{j, \varepsilon} \colon \mathbb{R}^{n} \to \mathbb{R}$ by
	\begin{align*}
		w_{j, \varepsilon} \left( x \right) \coloneqq x_{j} e^{- \varepsilon \left\lvert x \right\rvert^{2}}, \qquad x= \left( x_{1}, \ldots, x_{n} \right) \in \mathbb{R}^{n}.
	\end{align*}
	Then, we can see that $w_{j, \varepsilon} \in W^{2, \infty} \left( \mathbb{R}^{n} \right)$ and
	\begin{align*}
		\nabla w_{j, \varepsilon} \left( x \right) &=e^{- \varepsilon \left\lvert x \right\rvert^{2}} \left( -2 \varepsilon x_{j} x+e_{j} \right), \\
		\Delta w_{j, \varepsilon} \left( x \right) &=e^{- \varepsilon \left\lvert x \right\rvert^{2}} \left( 4 \varepsilon^{2} x_{j} \left\lvert x \right\rvert^{2} -2 \left( n+2 \right) \varepsilon x_{j} \right),
	\end{align*}
	whence follows
	\begin{align}
		\label{eq:3.a}
		\left\lVert \nabla w_{j, \varepsilon} \right\rVert_{\infty} &\leq 2 \sup_{\rho \geq 0} \rho e^{- \rho} +1 \leq 2, \\
		\label{eq:3.b}
		\left\lVert \Delta w_{j, \varepsilon} \right\rVert_{\infty} &\leq 4 \varepsilon^{\frac{1}{2}} \sup_{\rho \geq 0} \rho^{\frac{3}{2}} e^{- \rho} +2 \left( n+2 \right) \varepsilon^{\frac{1}{2}} \sup_{\rho \geq 0} \rho^{\frac{1}{2}} e^{- \rho} \leq \left( n+4 \right) \varepsilon^{\frac{1}{2}}.
	\end{align}

	Now, we show the case where $m=1$.
	Let $\varphi \in \left( L_{1}^{1} \cap L^{\infty} \right) \left( \mathbb{R}^{n} \right)$, $t>0$ and $\alpha \in \mathbb{Z}_{\geq 0}^{n}$ with $\left\lvert \alpha \right\rvert =1$.
	Then, there exists $j \in \left\{ 1, \ldots, n \right\}$ such that $\alpha =e_{j}$.
	Here and hereafter, let $C$ denote a positive constant independent of $t$ and $\varepsilon$ which may change from line to line.
	Multiplying \eqref{I} by $w_{j, \varepsilon}$ yields
	\begin{align*}
		w_{j, \varepsilon} u \left( t \right) &=w_{j, \varepsilon} e^{t \Delta} \varphi + \int_{0}^{t} w_{j, \varepsilon} e^{\left( t-s \right) \Delta} f \left( u \left( s \right) \right) ds \\
		&=e^{t \Delta} w_{j, \varepsilon} \varphi + \left( w_{j, \varepsilon} e^{t \Delta} \varphi -e^{t \Delta} w_{j, \varepsilon} \varphi \right) \\
		&\hspace{1cm} + \int_{0}^{t} e^{\left( t-s \right) \Delta} w_{j, \varepsilon} f \left( u \left( s \right) \right) ds+ \int_{0}^{t} \left( w_{j, \varepsilon} e^{\left( t-s \right) \Delta} f \left( u \left( s \right) \right) -e^{\left( t-s \right) \Delta} w_{j, \varepsilon} f \left( u \left( s \right) \right) \right) ds.
	\end{align*}
	By \eqref{eq:esti_f}, \eqref{eq:3.a}, \eqref{eq:3.b}, Proposition \ref{pro:P_global}, and Lemma \ref{lem:heat_weight_appro}, we have
	\begin{align*}
		\left\lVert w_{j, \varepsilon} u \left( t \right) \right\rVert_{1} &\leq \left\lVert e^{t \Delta} w_{j, \varepsilon} \varphi \right\rVert_{1} + \left\lVert w_{j, \varepsilon} e^{t \Delta} \varphi -e^{t \Delta} w_{j, \varepsilon} \varphi \right\rVert_{1} \\
		&\hspace{1cm} + \int_{0}^{t} \bigl\lVert e^{\left( t-s \right) \Delta} w_{j, \varepsilon} f \left( u \left( s \right) \right) \bigr\rVert_{1} ds \\
		&\hspace{1cm} + \int_{0}^{t} \bigl\lVert w_{j, \varepsilon} e^{\left( t-s \right) \Delta} f \left( u \left( s \right) \right) -e^{\left( t-s \right) \Delta} w_{j, \varepsilon} f \left( u \left( s \right) \right) \bigr\rVert_{1} ds \\
		&\leq \left\lVert w_{j, \varepsilon} \varphi \right\rVert_{1} + \left( \left\lVert \Delta w_{j, \varepsilon} \right\rVert_{\infty} t+ \left\lVert \nabla w_{j, \varepsilon} \right\rVert_{\infty} \left\lVert \nabla G_{1} \right\rVert_{1} t^{\frac{1}{2}} \right) \left\lVert \varphi \right\rVert_{1} \\
		&\hspace{1cm} +C \int_{0}^{t} \left\lVert u \left( s \right) \right\rVert_{\infty}^{p-1} \left\lVert w_{j, \varepsilon} u \left( s \right) \right\rVert_{1} ds \\
		&\hspace{1cm} +C \int_{0}^{t} \left( \left\lVert \Delta w_{j, \varepsilon} \right\rVert_{\infty} \left( t-s \right) + \left\lVert \nabla w_{j, \varepsilon} \right\rVert_{\infty} \left\lVert \nabla G_{1} \right\rVert_{1} \left( t-s \right)^{\frac{1}{2}} \right) \left\lVert u \left( s \right) \right\rVert_{p}^{p} ds \\
		&\leq \left\lVert x_{j} \varphi \right\rVert_{1} +C \left( \varepsilon^{\frac{1}{2}} t+t^{\frac{1}{2}} \right) \left\lVert \varphi \right\rVert_{1} \\
		&\hspace{1cm} +C \int_{0}^{t} \left( 1+s \right)^{- \frac{n}{2} \left( p-1 \right)} \left\lVert w_{j, \varepsilon} u \left( s \right) \right\rVert_{1} ds \\
		&\hspace{1cm} +C \int_{0}^{t} \left( \varepsilon^{\frac{1}{2}} \left( t-s \right) + \left( t-s \right)^{\frac{1}{2}} \right) \left( 1+s \right)^{- \frac{n}{2} \left( p-1 \right)} ds \\
		&= \xi_{\varepsilon} \left( t \right) + \int_{0}^{t} \eta \left( s \right) \left\lVert w_{j, \varepsilon} u \left( s \right) \right\rVert_{1} ds,
	\end{align*}
	where
	\begin{align*}
		\xi_{\varepsilon} \left( t \right) &\coloneqq \left\lVert x_{j} \varphi \right\rVert_{1} +C \left( \varepsilon^{\frac{1}{2}} t+t^{\frac{1}{2}} \right) \left\lVert \varphi \right\rVert_{1}+C \int_{0}^{t} \left( \varepsilon^{\frac{1}{2}} \left( t-s \right) + \left( t-s \right)^{\frac{1}{2}} \right) \left( 1+s \right)^{- \frac{n}{2} \left( p-1 \right)} ds, \\
		\eta \left( t \right) &\coloneqq C \left( 1+t \right)^{- \frac{n}{2} \left( p-1 \right)}.
	\end{align*}
	Therefore, from the Gr\"{o}nwall lemma and $\xi_{\varepsilon} \leq \xi_{1}$, we derive
	\begin{align}
		\label{eq:3.c}
		\left\lVert w_{j, \varepsilon} u \left( t \right) \right\rVert_{1} &\leq \xi_{\varepsilon} \left( t \right) + \int_{0}^{t} \xi_{\varepsilon} \left( s \right) \eta \left( s \right) \exp \left( \int_{s}^{t} \eta \left( \tau \right) d \tau \right) ds \nonumber \\
		&\leq \xi_{1} \left( t \right) + \int_{0}^{t} \xi_{1} \left( s \right) \eta \left( s \right) \exp \left( \int_{s}^{t} \eta \left( \tau \right) d \tau \right) ds.
	\end{align}
	In particular, since the right hand side on the last inequality in \eqref{eq:3.c} is finite and independent of $\varepsilon$, it follows from Fatou's lemma that $x_{j} u \left( t \right) \in L^{1} \left( \mathbb{R}^{n} \right)$, which in turn implies $x_{j} f \left( u \left( t \right) \right) \in L^{1} \left( \mathbb{R}^{n} \right)$.
	Moreover, by \eqref{I} and Theorem \ref{th:heat_commutator}, we obtain
	\begin{align*}
		x_{j} u \left( t \right) &=x_{j} e^{t \Delta} \varphi + \int_{0}^{t} x_{j} e^{\left( t-s \right) \Delta} f \left( u \left( s \right) \right) ds \\
		&=e^{t \Delta} x_{j} \varphi -2t \partial_{j} e^{t \Delta} \varphi + \int_{0}^{t} e^{\left( t-s \right) \Delta} x_{j} f \left( u \left( s \right) \right) ds-2 \int_{0}^{t} \left( t-s \right) \partial_{j} e^{\left( t-s \right) \Delta} f \left( u \left( s \right) \right) ds,
	\end{align*}
	whence follows $x_{j} u \in C \left( \left[ 0, + \infty \right); L^{1} \left( \mathbb{R}^{n} \right) \right)$.
	On the other hand, taking $\varepsilon \searrow 0$ in \eqref{eq:3.c} yields
	\begin{align*}
		\left\lVert x_{j} u \left( t \right) \right\rVert_{1} &\leq \xi_{0} \left( t \right) + \int_{0}^{t} \xi_{0} \left( s \right) \eta \left( s \right) \exp \left( \int_{s}^{t} \eta \left( \tau \right) d \tau \right) ds,
	\end{align*}
	where
	\begin{align*}
		\xi_{0} \left( t \right) &\coloneqq \left\lVert x_{j} \varphi \right\rVert_{1} +Ct^{\frac{1}{2}} \left\lVert \varphi \right\rVert_{1}+C \int_{0}^{t} \left( t-s \right)^{\frac{1}{2}} \left( 1+s \right)^{- \frac{n}{2} \left( p-1 \right)} ds.
	\end{align*}
	Since $p>p_{\mathrm{F}} \left( n \right)$, the integral appearing in the definition of $\xi_{0} \left( t \right)$ is estimated as
	\begin{align*}
		\int_{0}^{t} \left( t-s \right)^{\frac{1}{2}} \left( 1+s \right)^{- \frac{n}{2} \left( p-1 \right)} ds &\leq t^{\frac{1}{2}} \int_{0}^{t} \left( 1+s \right)^{- \frac{n}{2} \left( p-1 \right)} ds \leq Ct^{\frac{1}{2}}.
	\end{align*}
	Therefore, we have
	\begin{align*}
		&\int_{0}^{t} \xi_{0} \left( s \right) \eta \left( s \right) \exp \left( \int_{s}^{t} \eta \left( \tau \right) d \tau \right) ds \\
		&\hspace{1cm} \leq C \exp \left( \int_{0}^{+ \infty} \eta \left( \tau \right) d \tau \right) \int_{0}^{t} \left( 1+s^{\frac{1}{2}} \right) \eta \left( s \right) ds \\
		&\hspace{1cm} \leq C \exp \left( C \int_{0}^{+ \infty} \left( 1+ \tau \right)^{- \frac{n}{2} \left( p-1 \right)} d \tau \right) \left( 1+t^{\frac{1}{2}} \right) \int_{0}^{t} \left( 1+s \right)^{- \frac{n}{2} \left( p-1 \right)} ds \\
		&\hspace{1cm} \leq C \left( 1+t^{\frac{1}{2}} \right),
	\end{align*}
	whence follows
	\begin{align*}
		\left\lVert x_{j} u \left( t \right) \right\rVert_{1} \leq C \left( 1+t^{\frac{1}{2}} \right).
	\end{align*}

	Next, we assume that Theorem \ref{th:P_weight} holds for some $m \in \mathbb{Z}_{>0}$.
	Let $\varphi \in \left( L_{m+1}^{1} \cap L^{\infty} \right) \left( \mathbb{R}^{n} \right)$, $t>0$ and $\alpha' \in \mathbb{Z}_{\geq 0}^{n}$ with $\left\lvert \alpha' \right\rvert =m+1$.
	Then, there exist $\alpha \in \mathbb{Z}_{\geq 0}^{n}$ with $\left\lvert \alpha \right\rvert =m$ and $j \in \left\{ 1, \ldots, n \right\}$ such that $\alpha' = \alpha +e_{j}$.
	From the induction hypothesis and Remark \ref{rem:heat_remainder_esti}, it follows that
	\begin{align}
		\label{eq:3.e}
		\left\lVert x_{j} R_{\alpha} \left( t \right) \varphi \right\rVert_{1} &\leq C \left\{ t^{\frac{1}{2}} \left\lVert \left\lvert x \right\rvert^{m} \varphi \right\rVert_{1} + \left( t^{\frac{1}{2}} +t^{\frac{m+1}{2}} \right) \left\lVert \varphi \right\rVert_{1} \right\}, \\
		\label{eq:3.f}
		\left\lVert x_{j} R_{\alpha} \left( t-s \right) f \left( u \left( s \right) \right) \right\rVert_{1} &\leq C \left\{ \left( t-s \right)^{\frac{1}{2}} \left\lVert \left\lvert x \right\rvert^{m} f \left( u \left( s \right) \right) \right\rVert_{1} + \left( \left( t-s \right)^{\frac{1}{2}} + \left( t-s \right)^{\frac{m+1}{2}} \right) \left\lVert f \left( u \left( s \right) \right) \right\rVert_{1} \right\}
	\end{align}
	for any $s \in \left( 0, t \right)$, where $C$ does not depend on $s$.
	Multiplying \eqref{I} by $w_{j, \varepsilon} x^{\alpha}$ and applying Theorem \ref{th:heat_commutator}, we have
	\begin{align*}
		w_{j, \varepsilon} x^{\alpha} u \left( t \right) &=w_{j, \varepsilon} x^{\alpha} e^{t \Delta} \varphi + \int_{0}^{t} w_{j, \varepsilon} x^{\alpha} e^{\left( t-s \right) \Delta} f \left( u \left( s \right) \right) ds \\
		&=e^{t \Delta} w_{j, \varepsilon} x^{\alpha} \varphi + \left( w_{j, \varepsilon} e^{t \Delta} x^{\alpha} \varphi -e^{t \Delta} w_{j, \varepsilon} x^{\alpha} \varphi \right) + w_{j, \varepsilon} R_{\alpha} \left( t \right) \varphi \\
		&\hspace{1cm} + \int_{0}^{t} \left( w_{j, \varepsilon} e^{\left( t-s \right) \Delta} x^{\alpha} f \left( u \left( s \right) \right) -e^{\left( t-s \right) \Delta} w_{j, \varepsilon} x^{\alpha} f \left( u \left( s \right) \right) \right) ds \\
		&\hspace{1cm} + \int_{0}^{t} e^{\left( t-s \right) \Delta} w_{j, \varepsilon} x^{\alpha} f \left( u \left( s \right) \right) ds+ \int_{0}^{t} w_{j, \varepsilon} R_{\alpha} \left( t-s \right) f \left( u \left( s \right) \right) ds.
	\end{align*}
	By a computation similar to that in the case where $m=1$ with \eqref{eq:esti_f}, \eqref{eq:3.a}, \eqref{eq:3.b}, \eqref{eq:3.e}, \eqref{eq:3.f}, Proposition \ref{pro:P_global}, and Lemma \ref{lem:heat_weight_appro}, we can derive
	\begin{align*}
		\left\lVert w_{j, \varepsilon} x^{\alpha} u \left( t \right) \right\rVert_{1} \leq \widetilde{\xi}_{\varepsilon} \left( t \right) + \int_{0}^{t} \eta \left( s \right) \left\lVert w_{j, \varepsilon} x^{\alpha} u \left( s \right) \right\rVert_{1} ds,
	\end{align*}
	where
	\begin{align*}
		\widetilde{\xi}_{\varepsilon} \left( t \right) &\coloneqq C \left( 1+t^{\frac{1}{2}} + \varepsilon^{\frac{1}{2}} t+t^{\frac{m+1}{2}} \right) +C \int_{0}^{t} \left( \varepsilon^{\frac{1}{2}} \left( t-s \right) + \left( t-s \right)^{\frac{1}{2}} \right) \left( 1+s \right)^{- \frac{n}{2} \left( p-1 \right)} \left( 1+s^{\frac{m}{2}} \right) ds \\
		&\hspace{1cm} +C \int_{0}^{t} \left( \left( t-s \right)^{\frac{1}{2}} + \left( t-s \right)^{\frac{m+1}{2}} \right) \left( 1+s \right)^{- \frac{n}{2} \left( p-1 \right)} ds.
	\end{align*}
	Therefore, from the Gr\"{o}nwall lemma and $\widetilde{\xi}_{\varepsilon} \leq \widetilde{\xi}_{1}$, we have
	\begin{align}
		\label{eq:3.g}
		\left\lVert w_{j, \varepsilon} x^{\alpha} u \left( t \right) \right\rVert_{1} &\leq \widetilde{\xi}_{\varepsilon} \left( t \right) + \int_{0}^{t} \widetilde{\xi}_{\varepsilon} \left( s \right) \eta \left( s \right) \exp \left( \int_{s}^{t} \eta \left( \tau \right) d \tau \right) ds \nonumber \\
		&\leq \widetilde{\xi}_{1} \left( t \right) + \int_{0}^{t} \widetilde{\xi}_{1} \left( s \right) \eta \left( s \right) \exp \left( \int_{s}^{t} \eta \left( \tau \right) d \tau \right) ds.
	\end{align}
	In particular, since the right hand side on the last inequality in \eqref{eq:3.g} is finite and independent of $\varepsilon$, it follows from Fatou's lemma that $x^{\alpha'} u \left( t \right) =x_{j} x^{\alpha} u \left( t \right) \in L^{1} \left( \mathbb{R}^{n} \right)$, which in turn implies $x^{\alpha'} f \left( u \left( t \right) \right) \in L^{1} \left( \mathbb{R}^{n} \right)$.
	Furthermore, by \eqref{I} and Theorem \ref{th:heat_commutator}, we obtain
	\begin{align*}
		x^{\alpha'} u \left( t \right) &=x^{\alpha'} e^{t \Delta} \varphi + \int_{0}^{t} x^{\alpha'} e^{\left( t-s \right) \Delta} f \left( u \left( s \right) \right) ds \\
		&=e^{t \Delta} x^{\alpha'} \varphi +R_{\alpha'} \left( t \right) \varphi + \int_{0}^{t} e^{\left( t-s \right) \Delta} x^{\alpha'} f \left( u \left( s \right) \right) ds+ \int_{0}^{t} R_{\alpha'} \left( t-s \right) f \left( u \left( s \right) \right) ds,
	\end{align*}
	whence follows $x^{\alpha'} u \in C \left( \left[ 0, + \infty \right); L^{1} \left( \mathbb{R}^{n} \right) \right)$.
	On the other hand, taking $\varepsilon \searrow 0$ in \eqref{eq:3.g} yields
	\begin{align*}
		\bigl\lVert x^{\alpha'} u \left( t \right) \bigr\rVert_{1} &\leq \widetilde{\xi}_{0} \left( t \right) + \int_{0}^{t} \widetilde{\xi}_{0} \left( s \right) \eta \left( s \right) \exp \left( \int_{s}^{t} \eta \left( \tau \right) d \tau \right) ds,
	\end{align*}
	where
	\begin{align*}
		\widetilde{\xi}_{0} \left( t \right) &\coloneqq C \left( 1+t^{\frac{1}{2}} +t^{\frac{m+1}{2}} \right) +C \int_{0}^{t} \left( t-s \right)^{\frac{1}{2}} \left( 1+s \right)^{- \frac{n}{2} \left( p-1 \right)} \left( 1+s^{\frac{m}{2}} \right) ds \\
		&\hspace{1cm} +C \int_{0}^{t} \left( \left( t-s \right)^{\frac{1}{2}} + \left( t-s \right)^{\frac{m+1}{2}} \right) \left( 1+s \right)^{- \frac{n}{2} \left( p-1 \right)} ds.
	\end{align*}
	Since $p>p_{\mathrm{F}} \left( n \right)$, the integrals appearing in the definition of $\widetilde{\xi}_{0} \left( t \right)$ are estimated as
	\begin{align*}
		\int_{0}^{t} \left( t-s \right)^{\frac{1}{2}} \left( 1+s \right)^{- \frac{n}{2} \left( p-1 \right)} \left( 1+s^{\frac{m}{2}} \right) ds &\leq t^{\frac{1}{2}} \left( 1+t^{\frac{m}{2}} \right) \int_{0}^{t} \left( 1+s \right)^{- \frac{n}{2} \left( p-1 \right)} ds \\
		&\leq C \left( t^{\frac{1}{2}} +t^{\frac{m+1}{2}} \right), \\
		\int_{0}^{t} \left( \left( t-s \right)^{\frac{1}{2}} + \left( t-s \right)^{\frac{m+1}{2}} \right) \left( 1+s \right)^{- \frac{n}{2} \left( p-1 \right)} ds &\leq \left( t^{\frac{1}{2}} +t^{\frac{m+1}{2}} \right) \int_{0}^{t} \left( 1+s \right)^{- \frac{n}{2} \left( p-1 \right)} ds \\
		&\leq C \left( t^{\frac{1}{2}} +t^{\frac{m+1}{2}} \right).
	\end{align*}
	Therefore, we have
	\begin{align*}
		&\int_{0}^{t} \widetilde{\xi}_{0} \left( s \right) \eta \left( s \right) \exp \left( \int_{s}^{t} \eta \left( \tau \right) d \tau \right) ds \\
		&\hspace{1cm} \leq C \exp \left( \int_{0}^{+ \infty} \eta \left( \tau \right) d \tau \right) \int_{0}^{t} \left( 1+s^{\frac{1}{2}} +s^{\frac{m+1}{2}} \right) \eta \left( s \right) ds \\
		&\hspace{1cm} \leq C \exp \left( C \int_{0}^{+ \infty} \left( 1+ \tau \right)^{- \frac{n}{2} \left( p-1 \right)} d \tau \right) \left( 1+t^{\frac{1}{2}} +t^{\frac{m+1}{2}} \right) \int_{0}^{t} \left( 1+s \right)^{- \frac{n}{2} \left( p-1 \right)} ds \\
		&\hspace{1cm} \leq C \left( 1+t^{\frac{1}{2}} +t^{\frac{m+1}{2}} \right),
	\end{align*}
	whence follows
	\begin{align*}
		\bigl\lVert x^{\alpha'} u \left( t \right) \bigr\rVert_{1} &\leq C \left( 1+t^{\frac{1}{2}} +t^{\frac{m+1}{2}} \right) \leq C \left( 1+t^{\frac{m+1}{2}} \right).
	\end{align*}
	This completes the induction argument.
	\qed

\section{Proofs of Theorems \ref{th:P_appro} and \ref{th:P_aymptotics}} \label{sec:4}
We first prove Theorem \ref{th:P_appro} in the case where $N=1$.

\begin{prop} \label{pro:4.1}
	Under the same assumptions as in Theorem \ref{th:P_appro}, for any $q \in \left[ 1, + \infty \right]$, there exists $C_{q} >0$ such that the estimates
	\begin{align*}
		t^{\frac{n}{2} \left( 1- \frac{1}{q} \right)} \left\lVert u \left( t \right) -e^{t \Delta} \varphi_{1} \right\rVert_{q} \leq \begin{dcases}
			C_{q} t^{- \sigma} &\quad \text{if} \quad 0< \sigma <1, \\
			C_{q} t^{-1} \log \left( 1+t \right) &\quad \text{if} \quad \sigma =1, \\
			C_{q} t^{-1} &\quad \text{if} \quad \sigma >1
		\end{dcases}
	\end{align*}
	hold for all $t>1$, where
	\begin{align*}
		\sigma &\coloneqq \frac{n}{2} \left( p-1 \right) -1>0, \\
		\varphi_{1} &\coloneqq \varphi + \int_{0}^{+ \infty} f \left( u \left( s \right) \right) ds.
	\end{align*}
\end{prop}

\begin{proof}
	Let $q \in \left[ 1, + \infty \right]$ and let $t>1$.
	We divide the difference $u \left( t \right) -e^{t \Delta} \varphi_{1}$ into three parts by using \eqref{I}:
	\begin{align*}
		u \left( t \right) -e^{t \Delta} \varphi_{1}
		&= \int_{0}^{t/2} \left( e^{\left( t-s \right) \Delta} -e^{t \Delta} \right) f \left( u \left( s \right) \right) ds+ \int_{t/2}^{t} e^{\left( t-s \right) \Delta} f \left( u \left( s \right) \right) ds-e^{t \Delta} \int_{t/2}^{+ \infty} f \left( u \left( s \right) \right) ds \\
		&= \int_{0}^{t/2} \int_{0}^{1} \frac{d}{d \theta} \left( e^{\left( t-s \theta \right) \Delta} f \left( u \left( s \right) \right) \right) d \theta ds+ \int_{t/2}^{t} e^{\left( t-s \right) \Delta} f \left( u \left( s \right) \right) ds-e^{t \Delta} \int_{t/2}^{+ \infty} f \left( u \left( s \right) \right) ds \\
		&=- \int_{0}^{t/2} \int_{0}^{1} s \Delta e^{\left( t-s \theta \right) \Delta} f \left( u \left( s \right) \right) d \theta ds+ \int_{t/2}^{t} e^{\left( t-s \right) \Delta} f \left( u \left( s \right) \right) ds-e^{t \Delta} \int_{t/2}^{+ \infty} f \left( u \left( s \right) \right) ds.
	\end{align*}
	We note that since $p>p_{\mathrm{F}} \left( n \right) \Leftrightarrow \sigma >0$, we have
	\begin{align*}
		\frac{np}{2} - \frac{n}{2q} -1= \sigma + \frac{n}{2} \left( 1- \frac{1}{q} \right) >0.
	\end{align*}
	Therefore, from \eqref{eq:esti_f}, Proposition \ref{pro:P_global}, and Lemma \ref{lem:Lp-Lq}, we obtain
	\begin{align*}
		\left\lVert \int_{t/2}^{t} e^{\left( t-s \right) \Delta} f \left( u \left( s \right) \right) ds \right\rVert_{q}
		&\leq C \int_{t/2}^{t} \left\lVert u \left( s \right) \right\rVert_{pq}^{p} ds \\
		&\leq C \int_{t/2}^{t} s^{- \frac{n}{2} \left( 1- \frac{1}{pq} \right) p} ds \\
		&\leq Ct^{- \sigma - \frac{n}{2} \left( 1- \frac{1}{q} \right)}, \\
		\left\lVert e^{t \Delta} \int_{t/2}^{+ \infty} f \left( u \left( s \right) \right) ds \right\rVert_{q}
		&\leq C \int_{t/2}^{+ \infty} \left\lVert u \left( s \right) \right\rVert_{pq}^{p} ds \\
		&\leq C \int_{t/2}^{+ \infty} s^{- \frac{n}{2} \left( 1- \frac{1}{pq} \right) p} ds \\
		&\leq Ct^{- \sigma - \frac{n}{2} \left( 1- \frac{1}{q} \right)}.
	\end{align*}
	In the same way, we have
	\begin{align*}
		\left\lVert \int_{0}^{t/2} \int_{0}^{1} s \Delta e^{\left( t-s \theta \right) \Delta} f \left( u \left( s \right) \right) d \theta ds \right\rVert_{q}
		&\leq C \int_{0}^{t/2} \int_{0}^{1} s \left( t-s \theta \right)^{- \frac{n}{2} \left( 1- \frac{1}{q} \right) -1} \left\lVert u \left( s \right) \right\rVert_{p}^{p} d \theta ds \\
		&\leq C \int_{0}^{t/2} \int_{0}^{1} s \left( t-s \theta \right)^{- \frac{n}{2} \left( 1- \frac{1}{q} \right) -1} \left( 1+s \right)^{- \frac{n}{2} \left( p-1 \right)} d \theta ds \\
		&\leq Ct^{- \frac{n}{2} \left( 1- \frac{1}{q} \right) -1} \int_{0}^{t/2} \left( 1+s \right)^{- \sigma} ds.
	\end{align*}
	The integral appearing on the right hand side of the last inequality is estimated as
	\begin{align*}
		\int_{0}^{t/2} \left( 1+s \right)^{- \sigma} ds &\leq \begin{dcases}
			\frac{1}{1- \sigma} \left( 1+ \frac{t}{2} \right)^{1- \sigma} \leq \frac{1}{1- \sigma} \left( \frac{3t}{2} \right)^{1- \sigma}, &\quad 0< \sigma <1, \\
			\log \left( 1+ \frac{t}{2} \right) \leq \log \left( 1+t \right), &\quad \sigma =1, \\
			\frac{1}{\sigma -1}, &\quad \sigma >1.
		\end{dcases}
	\end{align*}
	As a consequence, we can deduce that
	\begin{align*}
		t^{\frac{n}{2} \left( 1- \frac{1}{q} \right)} \left\lVert u \left( t \right) -e^{t \Delta} \varphi_{1} \right\rVert_{q} \leq \begin{dcases}
			Ct^{- \sigma}, &\quad 0< \sigma <1, \\
			C \left( t^{-1} +t^{-1} \log \left( 1+t \right) \right) \leq Ct^{-1} \log \left( 1+t \right), &\quad \sigma =1, \\
			C \left( t^{- \sigma} +t^{-1} \right) \leq Ct^{-1}, &\quad \sigma >1.
		\end{dcases}
	\end{align*}
\end{proof}

Before we prove Theorem \ref{th:P_appro} in the case where $N \geq 2$, we prepare the following lemma:

\begin{lem} \label{lem:P_appro_decay}
	Under the same assumptions as in Theorem \ref{th:P_appro}, for any $N \in \mathbb{Z}_{>0}$ and $q \in \left[ 1, + \infty \right]$, there exists $A_{N, q} >0$ such that the estimate
	\begin{align*}
		\left\lVert u_{N} \left( t \right) \right\rVert_{q} \leq A_{N, q} \left( 1+t \right)^{- \frac{n}{2} \left( 1- \frac{1}{q} \right)}
	\end{align*}
	holds for all $t>0$.
\end{lem}

\begin{proof}
	We prove the assertion by induction on $N \in \mathbb{Z}_{>0}$.
	For the case where $N=1$, it follows from \eqref{eq:esti_f}, Proposition \ref{pro:P_global}, and Lemma \ref{lem:Lp-Lq_2} that
	\begin{align*}
		\left\lVert e^{t \Delta} \varphi_{1} \right\rVert_{q} &\leq \left\lVert e^{t \Delta} \varphi \right\rVert_{q} + \left\lVert e^{t \Delta} \int_{0}^{+ \infty} f \left( u \left( s \right) \right) ds \right\rVert_{q} \\
		&\leq C \left( 1+t \right)^{- \frac{n}{2} \left( 1- \frac{1}{q} \right)} \left( \left\lVert \varphi \right\rVert_{1} + \left\lVert \varphi \right\rVert_{\infty} \right) \\
		&\hspace{2cm} +C \left( 1+t \right)^{- \frac{n}{2} \left( 1- \frac{1}{q} \right)} \int_{0}^{+ \infty} \left\lVert u \left( s \right) \right\rVert_{\infty}^{p-1} \left( \left\lVert u \left( s \right) \right\rVert_{1} + \left\lVert u \left( s \right) \right\rVert_{\infty} \right) ds \\
		&\leq C \left( 1+t \right)^{- \frac{n}{2} \left( 1- \frac{1}{q} \right)} \left( \left\lVert \varphi \right\rVert_{1} + \left\lVert \varphi \right\rVert_{\infty} \right) +C \left( 1+t \right)^{- \frac{n}{2} \left( 1- \frac{1}{q} \right)} \int_{0}^{+ \infty} \left( 1+s \right)^{- \frac{n}{2} \left( p-1 \right)} ds \\
		&\leq C \left( 1+t \right)^{- \frac{n}{2} \left( 1- \frac{1}{q} \right)}
	\end{align*}
	for any $q \in \left[ 1, + \infty \right]$ and $t>0$.

	Next, we assume that Lemma \ref{lem:P_appro_decay} holds for some $N \in \mathbb{Z}_{>0}$ and show that it is also true when we replace $N$ by $N+1$.
	To do this, it suffices to consider the cases where $q=1$ and $q= + \infty$.
	In fact, if we can prove these cases, H\"{o}lder's inequality implies
	\begin{align*}
		\left\lVert u_{N+1} \left( t \right) \right\rVert_{q} &\leq \left\lVert u_{N+1} \left( t \right) \right\rVert_{1}^{\frac{1}{q}} \left\lVert u_{N+1} \left( t \right) \right\rVert_{\infty}^{1- \frac{1}{q}} \\
		&\leq \left( A_{N+1, 1} \right)^{\frac{1}{q}} \left( A_{N+1, \infty} \right)^{1- \frac{1}{q}} \left( 1+t \right)^{- \frac{n}{2} \left( 1- \frac{1}{q} \right)}
	\end{align*}
	for any $q \in \left( 1, + \infty \right)$ and $t>0$.
	Now, let $t>0$.
	We note that
	\begin{align*}
		u_{N+1} \left( t \right) &=e^{t \Delta} \left( \varphi + \int_{0}^{+ \infty} \left( f \left( u \left( s \right) \right) -f \left( u_{N} \left( s \right) \right) \right) ds \right) + \int_{0}^{t} e^{\left( t-s \right) \Delta} f \left( u_{N} \left( s \right) \right) ds \\
		&=e^{t \Delta} \varphi_{1} -e^{t \Delta} \int_{0}^{+ \infty} f \left( u_{N} \left( s \right) \right) ds+ \int_{0}^{t} e^{\left( t-s \right) \Delta} f \left( u_{N} \left( s \right) \right) ds.
	\end{align*}
	By \eqref{eq:esti_f} and Lemmas \ref{lem:Lp-Lq} and \ref{lem:Lp-Lq_2}, we have
	\begin{align*}
		\left\lVert u_{N+1} \left( t \right) \right\rVert_{1} &\leq \left\lVert e^{t \Delta} \varphi_{1} \right\rVert_{1} + \left\lVert e^{t \Delta} \int_{0}^{+ \infty} f \left( u_{N} \left( s \right) \right) ds \right\rVert_{1} + \int_{0}^{t} \bigl\lVert e^{\left( t-s \right)} f \left( u_{N} \left( s \right) \right) \bigr\rVert_{1} ds \\
		&\leq C+C \int_{0}^{+ \infty} \left\lVert u_{N} \left( s \right) \right\rVert_{p}^{p} ds+C \int_{0}^{t} \left\lVert u_{N} \left( s \right) \right\rVert_{p}^{p} ds \\
		&\leq C+C \int_{0}^{+ \infty} \left( 1+s \right)^{- \frac{n}{2} \left( p-1 \right)} ds \leq C, \\
		\left\lVert u_{N+1} \left( t \right) \right\rVert_{\infty} &\leq \left\lVert e^{t \Delta} \varphi_{1} \right\rVert_{\infty} + \left\lVert e^{t \Delta} \int_{0}^{+ \infty} f \left( u_{N} \left( s \right) \right) ds \right\rVert_{\infty} + \left( \int_{0}^{t/2} + \int_{t/2}^{t} \right) \bigl\lVert e^{\left( t-s \right)} f \left( u_{N} \left( s \right) \right) \bigr\rVert_{\infty} ds \\
		&\leq C \left( 1+t \right)^{- \frac{n}{2}} +C \left( 1+t \right)^{- \frac{n}{2}} \int_{0}^{+ \infty} \left\lVert u_{N} \left( s \right) \right\rVert_{\infty}^{p-1} \left( \left\lVert u_{N} \left( s \right) \right\rVert_{1} + \left\lVert u_{N} \left( s \right) \right\rVert_{\infty} \right) ds \\
		&\hspace{1cm} +C \int_{0}^{t/2} \left( 1+t-s \right)^{- \frac{n}{2}} \left\lVert u_{N} \left( s \right) \right\rVert_{\infty}^{p-1} \left( \left\lVert u_{N} \left( s \right) \right\rVert_{1} + \left\lVert u_{N} \left( s \right) \right\rVert_{\infty} \right) ds \\
		&\hspace{1cm} +C \int_{t/2}^{t} \left\lVert u_{N} \left( s \right) \right\rVert_{\infty}^{p} ds \\
		&\leq C \left( 1+t \right)^{- \frac{n}{2}} +C \left( 1+t \right)^{- \frac{n}{2}} \int_{0}^{+ \infty} \left( 1+s \right)^{- \frac{n}{2} \left( p-1 \right)} ds \\
		&\hspace{1cm} +C \int_{0}^{t/2} \left( 1+t-s \right)^{- \frac{n}{2}} \left( 1+s \right)^{- \frac{n}{2} \left( p-1 \right)} ds+C \int_{t/2}^{t} \left( 1+s \right)^{- \frac{np}{2}} ds \\
		&\leq C \left( 1+t \right)^{- \frac{n}{2}} +C \left( 1+t \right)^{- \frac{n}{2}} \int_{0}^{+ \infty} \left( 1+s \right)^{- \frac{n}{2} \left( p-1 \right)} ds \\
		&\hspace{1cm} +C \left( 1+t \right)^{- \frac{n}{2}} \int_{0}^{t/2} \left( 1+s \right)^{- \frac{n}{2} \left( p-1 \right)} ds+C \left( 1+t \right)^{- \frac{n}{2}} \int_{t/2}^{t} \left( 1+s \right)^{- \frac{n}{2} \left( p-1 \right)} ds \\
		&\leq C \left( 1+t \right)^{- \frac{n}{2}}.
	\end{align*}
	This completes the proof.
\end{proof}

\begin{proof}[Proof of Theorem \ref{th:P_appro}]
	We regard Theorem \ref{th:P_appro} as the assertion with respect to $N \in \mathbb{Z}_{>0}$:
	\begin{itemize}
		\item[(B)$_{N}$]
			For any $q \in \left[ 1, + \infty \right]$, there exists $C_{N, q} >0$ such that the estimate
		\begin{align*}
			t^{\frac{n}{2} \left( 1- \frac{1}{q} \right)} \left\lVert u \left( t \right) -u_{N} \left( t \right) \right\rVert_{q} \leq C_{N, q} \zeta_{N} \left( t \right)
		\end{align*}
		holds for all $t>2^{N-1}$, where
		\begin{align*}
			\zeta_{N} \left( t \right) \coloneqq \begin{dcases}
				t^{-N \sigma}, &\quad 0<N \sigma <1, \\
				t^{-1} \log \left( 1+t \right), &\quad N \sigma =1, \\
				t^{-1}, &\quad N \sigma >1.
			\end{dcases}
		\end{align*}
	\end{itemize}
	We show that (B)$_{N}$ is true for any $N \in \mathbb{Z}_{>0}$ by induction on $N$.
	We have already proved the case where $N=1$ in Proposition \ref{pro:4.1}.
	We assume that (B)$_{N}$ holds for some $N \in \mathbb{Z}_{>0}$.
	Let $q \in \left[ 1, + \infty \right]$ and let $t>2^{N}$.
	We divide the difference $u \left( t \right) -u_{N+1} \left( t \right)$ into three parts by using \eqref{I}:
	\begin{align*}
		u \left( t \right) -u_{N+1} \left( t \right)
		&= \int_{0}^{t/2} \left( e^{\left( t-s \right) \Delta} -e^{t \Delta} \right) \left( f \left( u \left( s \right) \right) - f \left( u_{N} \left( s \right) \right) \right) ds \\
		&\hspace{1cm} + \int_{t/2}^{t} e^{\left( t-s \right) \Delta} \left( f \left( u \left( s \right) \right) - f \left( u_{N} \left( s \right) \right) \right) ds \\
		&\hspace{1cm} -e^{t \Delta} \int_{t/2}^{+ \infty} \left( f \left( u \left( s \right) \right) - f \left( u_{N} \left( s \right) \right) \right) ds \\
		&=- \int_{0}^{t/2} \int_{0}^{1} s \Delta e^{\left( t-s \theta \right) \Delta} \left( f \left( u \left( s \right) \right) - f \left( u_{N} \left( s \right) \right) \right) d \theta ds \\
		&\hspace{1cm} + \int_{t/2}^{t} e^{\left( t-s \right) \Delta} \left( f \left( u \left( s \right) \right) - f \left( u_{N} \left( s \right) \right) \right) ds \\
		&\hspace{1cm} -e^{t \Delta} \int_{t/2}^{+ \infty} \left( f \left( u \left( s \right) \right) - f \left( u_{N} \left( s \right) \right) \right) ds.
	\end{align*}
	From \eqref{eq:esti_f}, (B)$_{N}$, Proposition \ref{pro:P_global}, and Lemma \ref{lem:P_appro_decay}, it follows that
	\begin{align*}
		&\left\lVert \int_{t/2}^{t} e^{\left( t-s \right) \Delta} \left( f \left( u \left( s \right) \right) - f \left( u_{N} \left( s \right) \right) \right) ds \right\rVert_{q} \\
		&\hspace{1cm} \leq C \int_{t/2}^{t} \left( \left\lVert u \left( s \right) \right\rVert_{\infty}^{p-1} + \left\lVert u_{N} \left( s \right) \right\rVert_{\infty}^{p-1} \right) \left\lVert u \left( s \right) -u_{N} \left( s \right) \right\rVert_{q} ds \\
		&\hspace{1cm} \leq C \int_{t/2}^{t} \left( 1+s \right)^{- \frac{n}{2} \left( p-1 \right)} s^{- \frac{n}{2} \left( 1- \frac{1}{q} \right)} \zeta_{N} \left( s \right) ds \\
		&\hspace{1cm} \leq Ct^{- \frac{n}{2} \left( 1- \frac{1}{q} \right)} \int_{t/2}^{t} s^{-1- \sigma} \zeta_{N} \left( s \right) ds, \\
		&\left\lVert e^{t \Delta} \int_{t/2}^{+ \infty} \left( f \left( u \left( s \right) \right) - f \left( u_{N} \left( s \right) \right) \right) ds \right\rVert_{q} \\
		&\hspace{1cm} \leq C \int_{t/2}^{+ \infty} \left( \left\lVert u \left( s \right) \right\rVert_{\infty}^{p-1} + \left\lVert u_{N} \left( s \right) \right\rVert_{\infty}^{p-1} \right) \left\lVert u \left( s \right) -u_{N} \left( s \right) \right\rVert_{q} ds \\
		&\hspace{1cm} \leq C \int_{t/2}^{+ \infty} \left( 1+s \right)^{- \frac{n}{2} \left( p-1 \right)} s^{- \frac{n}{2} \left( 1- \frac{1}{q} \right)} \zeta_{N} \left( s \right) ds \\
		&\hspace{1cm} \leq Ct^{- \frac{n}{2} \left( 1- \frac{1}{q} \right)} \int_{t/2}^{+ \infty} s^{-1- \sigma} \zeta_{N} \left( s \right) ds.
	\end{align*}
	In the same way, we obtain
	\begin{align*}
		&\left\lVert \int_{0}^{t/2} \int_{0}^{1} s \Delta e^{\left( t-s \theta \right) \Delta} \left( f \left( u \left( s \right) \right) -f \left( u_{N} \left( s \right) \right) \right) d \theta ds \right\rVert_{q} \\
		&\hspace{1cm} \leq C \int_{0}^{t/2} \int_{0}^{1} s \left( t-s \theta \right)^{- \frac{n}{2} \left( 1- \frac{1}{q} \right) -1} \left( \left\lVert u \left( s \right) \right\rVert_{\infty}^{p-1} + \left\lVert u_{N} \left( s \right) \right\rVert_{\infty}^{p-1} \right) \left\lVert u \left( s \right) -u_{N} \left( s \right) \right\rVert_{1} d \theta ds \\
		&\hspace{1cm} \leq Ct^{- \frac{n}{2} \left( 1- \frac{1}{q} \right) -1} \left( \int_{0}^{2^{N-1}} + \int_{2^{N-1}}^{t/2} \right) s \left( \left\lVert u \left( s \right) \right\rVert_{\infty}^{p-1} + \left\lVert u_{N} \left( s \right) \right\rVert_{\infty}^{p-1} \right) \left\lVert u \left( s \right) -u_{N} \left( s \right) \right\rVert_{1} ds \\
		&\hspace{1cm} \leq Ct^{- \frac{n}{2} \left( 1- \frac{1}{q} \right) -1} \left( \int_{0}^{2^{N-1}} \left( 1+s \right)^{- \frac{n}{2} \left( p-1 \right)} ds+ \int_{2^{N-1}}^{t/2} s \left( 1+s \right)^{- \frac{n}{2} \left( p-1 \right)} \zeta_{N} \left( s \right) ds \right) \\
		&\hspace{1cm} \leq Ct^{- \frac{n}{2} \left( 1- \frac{1}{q} \right) -1} \left( 1+ \int_{2^{N-1}}^{t/2} s^{- \sigma} \zeta_{N} \left( s \right) ds \right).
	\end{align*}
	Now, let $0< \varepsilon \ll 1$.
	Then, by simple calculations, we have
	\begin{align*}
		\int_{t/2}^{t} s^{-1- \sigma} \zeta_{N} \left( s \right) ds &\leq \int_{t/2}^{+ \infty} s^{-1- \sigma} \zeta_{N} \left( s \right) ds \\
		&\leq \begin{dcases}
			Ct^{- \left( N+1 \right) \sigma}, &\quad 0<N \sigma <1, \\
			Ct^{-1- \frac{1}{N} + \varepsilon}, &\quad N \sigma =1, \\
			Ct^{-1- \sigma}, &\quad N \sigma >1,
		\end{dcases} \\
		1+ \int_{2^{N-1}}^{t/2} s^{- \sigma} \zeta_{N} \left( s \right) ds
		&\leq \begin{dcases}
			1+ \int_{2^{N-1}}^{t/2} s^{- \left( N+1 \right) \sigma} ds, &\quad 0<N \sigma <1, \\
			1+C \int_{2^{N-1}}^{t/2} s^{-1- \frac{1}{N} + \varepsilon} ds, &\quad N \sigma =1, \\
			1+ \int_{2^{N-1}}^{t/2} s^{-1- \sigma} ds, &\quad N \sigma >1
		\end{dcases} \\
		&\leq \begin{dcases}
			1+Ct^{1- \left( N+1 \right) \sigma} \leq Ct^{1- \left( N+1 \right) \sigma}, &\quad 0< \left( N+1 \right) \sigma <1, \\
			1+ \log \frac{t}{2} \leq 2 \log \left( 1+t \right), &\quad \left( N+1 \right) \sigma =1, \\
			C, &\quad N \sigma <1< \left( N+1 \right) \sigma, \\
			C, &\quad N \sigma =1, \\
			C, &\quad N \sigma >1.
		\end{dcases}
	\end{align*}
	Combining these estimates yields
	\begin{align*}
		t^{\frac{n}{2} \left( 1- \frac{1}{q} \right)} \left\lVert u \left( t \right) -u_{N+1} \left( t \right) \right\rVert_{q} &\leq \begin{dcases}
			Ct^{- \left( N+1 \right) \sigma}, &\quad 0< \left( N+1 \right) \sigma <1, \\
			C \left( t^{-1} +t^{-1} \log \left( 1+t \right) \right) \leq Ct^{-1} \log \left( 1+t \right), &\quad \left( N+1 \right) \sigma =1, \\
			C \left( t^{- \left( N+1 \right) \sigma} +t^{-1} \right) \leq Ct^{-1}, &\quad N \sigma <1< \left( N+1 \right) \sigma, \\
			C \left( t^{-1- \frac{1}{N} + \varepsilon} +t^{-1} \right) \leq Ct^{-1}, &\quad N \sigma =1, \\
			C \left( t^{-1- \sigma} +t^{-1} \right) \leq Ct^{-1}, &\quad N \sigma >1.
		\end{dcases}
	\end{align*}
	This completes the proof.
\end{proof}

\begin{proof}[Proof of Theorem \ref{th:P_aymptotics}]
	Let $q \in \left[ 1, + \infty \right]$.
	We note that
	\begin{align*}
		p>1+ \frac{3+m}{n} \quad \Longleftrightarrow \quad - \frac{n}{2} \left( p-1 \right) + \frac{m+1}{2} <-1.
	\end{align*}
	By \eqref{eq:esti_f}, Proposition \ref{pro:P_global}, and Theorem \ref{th:P_weight}, we have
	\begin{align*}
		\left\lVert x^{\alpha} \varphi_{1} \right\rVert_{1} &\leq \left\lVert x^{\alpha} \varphi \right\rVert_{1} + \int_{0}^{+ \infty} \left\lVert x^{\alpha} f \left( u \left( s \right) \right) \right\rVert_{1} ds \\
		&\leq \left\lVert x^{\alpha} \varphi \right\rVert_{1} +C \int_{0}^{+ \infty} \left\lVert u \left( s \right) \right\rVert_{\infty}^{p-1} \left\lVert x^{\alpha} u \left( s \right) \right\rVert_{1} ds \\
		&\leq \left\lVert x^{\alpha} \varphi \right\rVert_{1} +C \int_{0}^{+ \infty} \left( 1+s \right)^{- \frac{n}{2} \left( p-1 \right) + \frac{\left\lvert \alpha \right\rvert}{2}} ds \leq C
	\end{align*}
	for all $\alpha \in \mathbb{Z}_{\geq 0}^{n}$ with $\left\lvert \alpha \right\rvert \leq m+1$, whence follows $\varphi_{1} \in L_{m+1}^{1} \left( \mathbb{R}^{n} \right)$.
	Therefore, it follows from Proposition \ref{pro:heat_asymptotics_higher} that
	\begin{align}
		\label{eq:4.1}
		&t^{\frac{n}{2} \left( 1- \frac{1}{q} \right)} {\,} \biggl\lVert e^{t \Delta} \varphi_{1} - \sum_{k=0}^{m} 2^{-k} t^{- \frac{k}{2}} \sum_{\left\lvert \alpha \right\rvert =k} c_{\alpha} \delta_{t} \left( \bm{h}_{\alpha} G_{1} \right) \biggr\rVert_{q} \leq 2^{- \left( m+1 \right)} t^{- \frac{m+1}{2}} \sum_{\left\lvert \alpha \right\rvert =m+1} \frac{1}{\alpha !} \left\lVert \bm{h}_{\alpha} G_{1} \right\rVert_{q} \left\lVert x^{\alpha} \varphi_{1} \right\rVert_{1}
	\end{align}
	holds for any $t>0$, where
	\begin{align*}
		c_{\alpha} \coloneqq \frac{1}{\alpha !} \int_{\mathbb{R}^{n}} y^{\alpha} \varphi_{1} \left( y \right) dy.
	\end{align*}
	On the other hand, by Theorem \ref{th:P_appro} with $N=1$, there exists $C_{1, q} >0$ such that
	\begin{align}
		\label{eq:4.2}
		t^{\frac{n}{2} \left( 1- \frac{1}{q} \right)} \left\lVert u \left( t \right) -e^{t \Delta} \varphi_{1} \right\rVert_{q} \leq \begin{dcases}
			C_{1, q} t^{- \frac{1}{2}}, &\quad m=0, \\
			C_{1, q} t^{-1}, &\quad m=1
		\end{dcases}
	\end{align}
	hold for any $t>1$.
	Here, we remark that
	\begin{align*}
		\sigma \coloneqq \frac{n}{2} \left( p-1 \right) -1> \frac{m+1}{2} = \begin{dcases}
			\frac{1}{2}, &\quad m=0, \\
			1, &\quad m=1.
		\end{dcases}
	\end{align*}
	Combining \eqref{eq:4.1} and \eqref{eq:4.2}, we can deduce the desired result.
\end{proof}

\begin{rem} \label{rem:P_asymptotics_higher}
	In the same way as in the proof of Theorem \ref{th:P_aymptotics}, we see that if $m \in \mathbb{Z}_{\geq 0}$, $p>1+ \left( 3+m \right) /n$ and $\varphi \in \left( L_{m+1}^{1} \cap L^{\infty} \right) \left( \mathbb{R}^{n} \right)$, then for any $q \in \left[ 1, + \infty \right]$, there exists $C_{q} >0$ such that the estimates
	\begin{align*}
		t^{\frac{n}{2} \left( 1- \frac{1}{q} \right)} \biggl\lVert u \left( t \right) - \sum_{k=0}^{m} 2^{-k} t^{- \frac{k}{2}} \sum_{\left\lvert \alpha \right\rvert =k} c_{\alpha} \delta_{t} \left( \bm{h}_{\alpha} G_{1} \right) \biggr\rVert_{q} &\leq \begin{dcases}
			C_{q} t^{- \frac{1}{2}}, &\qquad m=0, \\
			C_{q} t^{-1}, &\qquad m \geq 1
		\end{dcases}
	\end{align*}
	hold for all $t>1$, where
	\begin{align*}
		c_{\alpha} &\coloneqq \frac{1}{\alpha !} \int_{\mathbb{R}^{n}} y^{\alpha} \varphi_{1} \left( y \right) dy, \\
		\varphi_{1} &\coloneqq \varphi + \int_{0}^{+ \infty} f \left( u \left( s \right) \right) ds.
	\end{align*}
	When $m \geq 2$, it is not suitable to consider the above estimates as higher order asymptotic expansions of the global solution since some terms in the asymptotic profiles decay faster than the remainders as $t \to + \infty$.
	See also Corollary \ref{cor:B.3} in Appendix \ref{app:B}.
\end{rem}

\appendix
\section{Proof of Proposition \ref{pro:P_global}} \label{app:A}

	Let $\varepsilon_{0} >0$ and let $\varphi \in \left( L^{1} \cap L^{\infty} \right) \left( \mathbb{R}^{n} \right)$ with $\left\lVert \varphi \right\rVert_{1} + \left\lVert \varphi \right\rVert_{\infty} \leq \varepsilon_{0}$.
	First, we prove the existence of global solutions to \eqref{P}.
	For each $M>0$, we define
	\begin{align*}
		X_{M} &\coloneqq \left\{ u \in X; {\,} \left\lVert u \right\rVert_{X} \leq M \right\}, \\
		\left\lVert u \right\rVert_{X} &\coloneqq \sup_{t \geq 0} \left( \left\lVert u \left( t \right) \right\rVert_{1} + \left( 1+t \right)^{\frac{n}{2}} \left\lVert u \left( t \right) \right\rVert_{\infty} \right), \\
		d \left( u, v \right) &\coloneqq \left\lVert u-v \right\rVert_{X}.
	\end{align*}
	Then, we can see that $\left( X_{M}, d \right)$ is a complete metric space.
	In addition, we define a mapping $\Phi \colon X_{M} \to X$ by
	\begin{align*}
		\left( \Phi u \right) \left( t \right) =e^{t \Delta} \varphi + \int_{0}^{t} e^{\left( t-s \right) \Delta} f \left( u \left( s \right) \right) ds.
	\end{align*}
	We show that $\Phi$ is a contraction on $X_{M}$ by choosing $M>0$ and $\varepsilon_{0} >0$ sufficiently small.
	If we can prove this assertion, by applying the Banach fixed point theorem, we obtain a global solution $u \in X_{M}$ to \eqref{P} as a fixed point of $\Phi$.
	Furthermore, H\"{o}lder's inequality implies
	\begin{align*}
		\left\lVert u \left( t \right) \right\rVert_{q} \leq \left\lVert u \left( t \right) \right\rVert_{1}^{\frac{1}{q}} \left\lVert u \left( t \right) \right\rVert_{\infty}^{1- \frac{1}{q}} \leq M \left( 1+t \right)^{- \frac{n}{2} \left( 1- \frac{1}{q} \right)}
	\end{align*}
	for any $q \in \left[ 1, + \infty \right]$ and $t \geq 0$.

	We first show that $\Phi \left( X_{M} \right) \subset X_{M}$ for appropriate $M>0$ and $\varepsilon_{0} >0$.
	Let $u \in X_{M}$ and let $t>0$.
	Then, by \eqref{eq:esti_f} and Lemmas \ref{lem:Lp-Lq} and \ref{lem:Lp-Lq_2}, we have
	\begin{align*}
		\left\lVert \left( \Phi u \right) \left( t \right) \right\rVert_{1} &\leq \left\lVert e^{t \Delta} \varphi \right\rVert_{1} + \int_{0}^{t} \bigl\lVert e^{\left( t-s \right) \Delta} f \left( u \left( s \right) \right) \bigr\rVert_{1} ds \\
		&\leq \left\lVert \varphi \right\rVert_{1} +C \int_{0}^{t} \left\lVert u \left( s \right) \right\rVert_{\infty}^{p-1} \left\lVert u \left( s \right) \right\rVert_{1} ds \\
		&\leq \varepsilon_{0} +C \left\lVert u \right\rVert_{X}^{p} \int_{0}^{t} \left( 1+s \right)^{- \frac{n}{2} \left( p-1 \right)} ds \\
		&\leq \varepsilon_{0} +CM^{p}, \\
		\left\lVert \left( \Phi u \right) \left( t \right) \right\rVert_{\infty} &\leq \left\lVert e^{t \Delta} \varphi \right\rVert_{\infty} + \left( \int_{0}^{t/2} + \int_{t/2}^{t} \right) \bigl\lVert e^{\left( t-s \right) \Delta} f \left( u \left( s \right) \right) \bigr\rVert_{\infty} ds \\
		&\leq C \left( 1+t \right)^{- \frac{n}{2}} \left( \left\lVert \varphi \right\rVert_{1} + \left\lVert \varphi \right\rVert_{\infty} \right) +C \int_{t/2}^{t} \left\lVert u \left( s \right) \right\rVert_{\infty}^{p} ds \\
		&\hspace{1cm} +C \int_{0}^{t/2} \left( 1+t-s \right)^{- \frac{n}{2}} \left\lVert u \left( s \right) \right\rVert_{\infty}^{p-1} \left( \left\lVert u \left( s \right) \right\rVert_{1} + \left\lVert u \left( s \right) \right\rVert_{\infty} \right) ds \\
		&\leq C \left( 1+t \right)^{- \frac{n}{2}} \varepsilon_{0} +C \left\lVert u \right\rVert_{X}^{p} \int_{t/2}^{t} \left( 1+s \right)^{- \frac{np}{2}} ds \\
		&\hspace{1cm} +C \left\lVert u \right\rVert_{X}^{p} \int_{0}^{t/2} \left( 1+t-s \right)^{- \frac{n}{2}} \left( 1+s \right)^{- \frac{n}{2} \left( p-1 \right)} ds \\
		&\leq C \left( 1+t \right)^{- \frac{n}{2}} \varepsilon_{0} +CM^{p} \left( 1+t \right)^{- \frac{n}{2}} \int_{t/2}^{t} \left( 1+s \right)^{- \frac{n}{2} \left( p-1 \right)} ds \\
		&\hspace{1cm} +CM^{p} \left( 1+t \right)^{- \frac{n}{2}} \int_{0}^{t/2} \left( 1+s \right)^{- \frac{n}{2} \left( p-1 \right)} ds \\
		&\leq C \left( 1+t \right)^{- \frac{n}{2}} \varepsilon_{0} +CM^{p} \left( 1+t \right)^{- \frac{n}{2}}.
	\end{align*}
	Combining these estimates yields
	\begin{align*}
		\left\lVert \Phi u \right\rVert_{X} \leq C' \varepsilon_{0} +C' M^{p}.
	\end{align*}
	Therefore, by taking $M>0$ and $\varepsilon_{0} >0$ to satisfy
	\begin{align*}
		C' M^{p-1} \leq \frac{1}{4}, \qquad \varepsilon_{0} \leq \frac{M}{2C'},
	\end{align*}
	we obtain $\left\lVert \Phi u \right\rVert_{X} \leq M$, whence follows $\Phi u \in X_{M}$.
	In the same way, we can derive
	\begin{align*}
		\left\lVert \Phi u- \Phi v \right\rVert_{X} \leq 2C' M^{p-1} \left\lVert u-v \right\rVert_{X} \leq \frac{1}{2} \left\lVert u-v \right\rVert_{X}
	\end{align*}
	for any $u, v \in X_{M}$, which implies that the mapping $\Phi \colon X_{M} \to X_{M}$ is a contraction.

	Finally, we show the uniqueness of global solutions to \eqref{P}.
	Let $u, v \in X$ be global solutions to \eqref{P} and set
	\begin{align*}
		M' \coloneqq \left\lVert u \right\rVert_{L^{\infty} \left( 0, + \infty; L^{\infty} \right)} + \left\lVert v \right\rVert_{L^{\infty} \left( 0, + \infty; L^{\infty} \right)}.
	\end{align*}
	Then, $u$ and $v$ satisfy the following integral equations, respectively:
	\begin{align*}
		u \left( t \right) &=e^{t \Delta} \varphi + \int_{0}^{t} e^{\left( t-s \right) \Delta} f \left( u \left( s \right) \right) ds, \\
		v \left( t \right) &=e^{t \Delta} \varphi + \int_{0}^{t} e^{\left( t-s \right) \Delta} f \left( v \left( s \right) \right) ds.
	\end{align*}
	Furthermore, by \eqref{eq:esti_f} and Lemma \ref{lem:Lp-Lq}, we have
	\begin{align*}
		\left\lVert u \left( t \right) -v \left( t \right) \right\rVert_{1}
		&\leq C \int_{0}^{t} \left( \left\lVert u \left( s \right) \right\rVert_{\infty}^{p-1} + \left\lVert v \left( s \right) \right\rVert_{\infty}^{p-1} \right) \left\lVert u \left( s \right) -v \left( s \right) \right\rVert_{1} ds \\
		&\leq C {M'}^{p-1} \int_{0}^{t} \left\lVert u \left( s \right) -v \left( s \right) \right\rVert_{1} ds
	\end{align*}
	for any $t>0$.
	Applying the Gr\"{o}nwall lemma, we obtain
	\begin{align*}
		\left\lVert u \left( t \right) -v \left( t \right) \right\rVert_{1} =0,
	\end{align*}
	whence follows $u=v$.

\section{Convergence to the asymptotic expansions} \label{app:B}

We are not able to apply Proposition \ref{pro:heat_asymptotics_higher} for $\varphi \in L_{m}^{1} \left( \mathbb{R}^{n} \right)$ to obtain the $m$-th order asymptotic expansion with a bound of the remainder with decay rate in $t$.
For $\varphi \in L_{m}^{1} \left( \mathbb{R}^{n} \right)$, we show that the remainder vanishes with higher order than $- \frac{n}{2} {\,} \bigl( 1- \frac{1}{q} \bigr)$ as $t \to + \infty$.

\begin{prop} \label{pro:B.1}
	Let $m \in \mathbb{Z}_{\geq 0}$, $\varphi \in L_{m}^{1} \left( \mathbb{R}^{n} \right)$ and $q \in \left[ 1, + \infty \right]$.
	Then,
	\begin{align*}
		\lim_{t \to + \infty} t^{\frac{n}{2} \left( 1- \frac{1}{q} \right)} {\,} \biggl\lVert e^{t \Delta} \varphi - \sum_{k=0}^{m} 2^{-k} t^{- \frac{k}{2}} \sum_{\left\lvert \alpha \right\rvert =k} c_{\alpha} \delta_{t} \left( \bm{h}_{\alpha} G_{1} \right) \biggr\rVert_{q} =0,
	\end{align*}
	where
	\begin{align*}
		c_{\alpha} \coloneqq \frac{1}{\alpha !} \int_{\mathbb{R}^{n}} y^{\alpha} \varphi \left( y \right) dy.
	\end{align*}
\end{prop}

\begin{rem}
	Proposition 1.1 in \cite{Kobayashi-Misawa} follows from Proposition \ref{pro:B.1} with $m=0$ and $n=q=2$.
	The authors in \cite{Kobayashi-Misawa} showed the result with the aid of the Fourier transform.
	More generally, for any $\varphi \in L^{1} \left( \mathbb{R}^{n} \right)$ and $q \in \left[ 1, + \infty \right]$, Proposition \ref{pro:B.1} with $m=0$ implies
	\begin{align*}
		\lim_{t \to + \infty} t^{\frac{n}{2} \left( 1- \frac{1}{q} \right)} \left\lVert e^{t \Delta} \varphi \right\rVert_{q} &= \left\lvert c_{0} \right\rvert \left\lVert G_{1} \right\rVert_{q} \\
		&= \begin{cases}
			\left\lvert c_{0} \right\rvert \left( 4 \pi \right)^{- \frac{n}{2} \left( 1- \frac{1}{q} \right)} q^{- \frac{n}{2q}}, &\quad 1 \leq q<+ \infty, \\
			\left\lvert c_{0} \right\rvert \left( 4 \pi \right)^{- \frac{n}{2}}, &\quad q=+ \infty,
		\end{cases}
	\end{align*}
	where
	\begin{align*}
		c_{0} \coloneqq \int_{\mathbb{R}^{n}} \varphi \left( y \right) dy.
	\end{align*}
\end{rem}

\begin{proof}[Proof of Proposition \ref{pro:B.1}]
	To begin with, $C_{c} \left( \mathbb{R}^{n} \right)$ denotes the set of continuous real-valued functions on $\mathbb{R}^{n}$ with compact support.
	Since $\varphi \in L^{1} \left( \mathbb{R}^{n} \right)$ and $C_{c} \left( \mathbb{R}^{n} \right)$ is dense in $L^{1} \left( \mathbb{R}^{n} \right)$, there exists a sequence $\left( \varphi_{j}; j \in \mathbb{Z}_{>0} \right)$ in $C_{c} \left( \mathbb{R}^{n} \right)$ such that $\left\lVert \varphi_{j} - \varphi \right\rVert_{1} \to 0$ as $j \to + \infty$.
	In particular, $\varphi_{j} \in L_{m+1}^{1} \left( \mathbb{R}^{n} \right)$ for any $j \in \mathbb{Z}_{>0}$.
	Here, we set
	\begin{align*}
		c_{j, \alpha} \coloneqq \frac{1}{\alpha !} \int_{\mathbb{R}^{n}} y^{\alpha} \varphi_{j} \left( y \right) dy
	\end{align*}
	for each $j \in \mathbb{Z}_{>0}$ and $\alpha \in \mathbb{Z}_{\geq 0}^{n}$ with $\left\lvert \alpha \right\rvert \leq m$.
	Then, by Lemma \ref{lem:Lp-Lq} and Proposition \ref{pro:heat_asymptotics_higher}, we have
	\begin{align*}
		&t^{\frac{n}{2} \left( 1- \frac{1}{q} \right)} {\,} \biggl\lVert e^{t \Delta} \varphi - \sum_{k=0}^{m} 2^{-k} t^{- \frac{k}{2}} \sum_{\left\lvert \alpha \right\rvert =k} c_{\alpha} \delta_{t} \left( \bm{h}_{\alpha} G_{1} \right) \biggr\rVert_{q} \\
		&\hspace{1cm} \leq t^{\frac{n}{2} \left( 1- \frac{1}{q} \right)} \left\lVert e^{t \Delta} \left( \varphi - \varphi_{j} \right) \right\rVert_{q} +t^{\frac{n}{2} \left( 1- \frac{1}{q} \right)} {\,} \biggl\lVert e^{t \Delta} \varphi_{j} - \sum_{k=0}^{m} 2^{-k} t^{- \frac{k}{2}} \sum_{\left\lvert \alpha \right\rvert =k} c_{j, \alpha} \delta_{t} \left( \bm{h}_{\alpha} G_{1} \right) \biggr\rVert_{q} \\
		&\hspace{2cm} +t^{\frac{n}{2} \left( 1- \frac{1}{q} \right)} \sum_{k=0}^{m} 2^{-k} t^{- \frac{k}{2}} \sum_{\left\lvert \alpha \right\rvert =k} \left\lvert c_{j, \alpha} -c_{\alpha} \right\rvert \left\lVert \delta_{t} \left( \bm{h}_{\alpha} G_{1} \right) \right\rVert_{q} \\
		&\hspace{1cm} \leq \left\lVert G_{1} \right\rVert_{q} \left\lVert \varphi - \varphi_{j} \right\rVert_{1} +2^{- \left( m+1 \right)} t^{- \frac{m+1}{2}} \sum_{\left\lvert \alpha \right\rvert =m+1} \frac{1}{\alpha !} \left\lVert \bm{h}_{\alpha} G_{1} \right\rVert_{q} \left\lVert x^{\alpha} \varphi_{j} \right\rVert_{1} \\
		&\hspace{2cm} + \left\lVert G_{1} \right\rVert_{q} \left\lVert \varphi_{j} - \varphi \right\rVert_{1} + \sum_{k=1}^{m} 2^{-k} t^{- \frac{k}{2}} \sum_{\left\lvert \alpha \right\rvert =k} \frac{1}{\alpha !} \left( \left\lVert x^{\alpha} \varphi_{j} \right\rVert_{1} + \left\lVert x^{\alpha} \varphi \right\rVert_{1} \right) \left\lVert \bm{h}_{\alpha} G_{1} \right\rVert_{q}
	\end{align*}
	for any $t>0$, whence follows
	\begin{align*}
		\limsup_{t \to + \infty} t^{\frac{n}{2} \left( 1- \frac{1}{q} \right)} {\,} \biggl\lVert e^{t \Delta} \varphi - \sum_{k=0}^{m} 2^{-k} t^{- \frac{k}{2}} \sum_{\left\lvert \alpha \right\rvert =k} c_{\alpha} \delta_{t} \left( \bm{h}_{\alpha} G_{1} \right) \biggr\rVert_{q} \leq 2 \left\lVert G_{1} \right\rVert_{q} \left\lVert \varphi_{j} - \varphi \right\rVert_{1}.
	\end{align*}
	Furthermore, since $\left\lVert \varphi_{j} - \varphi \right\rVert_{1} \to 0$ as $j \to + \infty$, we can deduce that
	\begin{align*}
		\lim_{t \to + \infty} t^{\frac{n}{2} \left( 1- \frac{1}{q} \right)} {\,} \biggl\lVert e^{t \Delta} \varphi - \sum_{k=0}^{m} 2^{-k} t^{- \frac{k}{2}} \sum_{\left\lvert \alpha \right\rvert =k} c_{\alpha} \delta_{t} \left( \bm{h}_{\alpha} G_{1} \right) \biggr\rVert_{q} =0.
	\end{align*}
\end{proof}

By using Proposition \ref{pro:B.1} instead of Proposition \ref{pro:heat_asymptotics_higher} in the proof of Theorem \ref{th:P_aymptotics}, we have the following corollary:

\begin{cor} \label{cor:B.3}
	Let $m \in \mathbb{Z}_{\geq 0}$ and let $p>1+ \left( 2+m \right) /n$.
	Let $\varphi \in \left( L_{m}^{1} \cap L^{\infty} \right) \left( \mathbb{R}^{n} \right)$ satisfy $\left\lVert \varphi \right\rVert_{1} + \left\lVert \varphi \right\rVert_{\infty} \leq \varepsilon_{0}$ and let $u \in X$ be the global solution to \eqref{P} given in Proposition \ref{pro:P_global}.
	Then,
	\begin{align*}
		\lim_{t \to + \infty} t^{\frac{n}{2} \left( 1- \frac{1}{q} \right)} {\,} \biggl\lVert u \left( t \right) - \sum_{k=0}^{m} 2^{-k} t^{- \frac{k}{2}} \sum_{\left\lvert \alpha \right\rvert =k} c_{\alpha} \delta_{t} \left( \bm{h}_{\alpha} G_{1} \right) \biggr\rVert_{q} =0
	\end{align*}
	holds for any $q \in \left[ 1, + \infty \right]$, where
	\begin{align*}
		c_{\alpha} &\coloneqq \frac{1}{\alpha !} \int_{\mathbb{R}^{n}} y^{\alpha} \varphi_{1} \left( y \right) dy, \\
		\varphi_{1} &\coloneqq \varphi + \int_{0}^{+ \infty} f \left( u \left( s \right) \right) ds.
	\end{align*}
\end{cor}



\end{document}